\documentclass[12pt]{amsart}
\usepackage{amsmath,amsfonts,amsthm,amscd,amssymb,mathrsfs,amssymb}
\usepackage{stmaryrd}
\usepackage{graphicx}
\newtheorem{theorem}{Theorem}

\newtheorem{proposition}[theorem]{Proposition}

\newtheorem{definition}[theorem]{Definition}

\newtheorem{corollary}[theorem]{Corollary}

\newtheorem{remark}[theorem]{Remark}

\newcommand{\CP}{\mathbb{CP}}
\newcommand{\CC}{\mathbb{C}}

\newcommand{\RR}{\mathbb{R}}
\newcommand{\ZZ}{\mathbb{Z}}

\numberwithin{equation}{section}
\numberwithin{theorem}{section}
\numberwithin{table}{section}
\numberwithin{table}{section}
\begin{document}
\bibliographystyle{amsalpha} 
\title[Anti-self-dual orbifolds]{Anti-self-dual orbifolds with cyclic\\ quotient singularities}
\author{Michael T. Lock}
\address{Department of Mathematics, University of Wisconsin, Madison, 
WI, 53706}
\email{lock@math.wisc.edu}
\author{Jeff A. Viaclovsky}
\address{Department of Mathematics, University of Wisconsin, Madison, 
WI, 53706}
\email{jeffv@math.wisc.edu}
\thanks{Research partially supported by NSF Grants DMS-0804042 and DMS-1105187}
\begin{abstract}
An index theorem for the anti-self-dual deformation complex on anti-self-dual orbifolds
with cyclic quotient singularities is proved. We present two applications of this theorem. The first is to compute the dimension of the deformation space of the Calderbank-Singer scalar-flat K\"ahler toric ALE spaces. A corollary of this is that, except for the Eguchi-Hanson metric, all of these spaces admit non-toric anti-self-dual deformations, thus yielding many new examples of anti-self-dual ALE spaces. For our second application,  we compute the dimension of the deformation space of the canonical Bochner-K\"ahler metric on any weighted projective space $\CP^2_{(r,q,p)}$ for relatively prime integers $1 < r < q < p$. A corollary of this is that, while these metrics are rigid as Bochner-K\"ahler metrics, infinitely many of these admit non-trival self-dual deformations, yielding a large class of new examples of self-dual orbifold metrics on certain weighted projective spaces.
\end{abstract}
\date{May 17, 2012}
\maketitle
\setcounter{tocdepth}{1}
\tableofcontents

\section{Introduction}
\label{intro}
If $(M^{4},g)$ is an {\em{oriented}} four-dimensional Riemannian manifold, 
the Hodge star operator $*:\Lambda^{2} \mapsto\Lambda^{2}$ satisfies $*^2 = Id$, 
and induces the decomposition on the space of $2$-forms
$\Lambda^{2}=\Lambda^{2}_{+} \oplus\Lambda^{2}_{-}$,
where $\Lambda^{2}_{\pm}$ are the $\pm 1$ eigenspaces of $*$. 
The Weyl tensor can be viewed as an operator 
$\mathcal{W}_g: \Lambda^2 \rightarrow \Lambda^2$, so this 
decomposition enables us to decompose the Weyl tensor 
as $\mathcal{W}_g = \mathcal{W}^{+}_g + \mathcal{W}^{-}_g$,
into the self-dual and anti-self-dual Weyl tensors, respectively. 
The metric $g$ is called {\em{anti-self-dual}} if $\mathcal{W}^+_g = 0$,
and $g$ is called {\em{self-dual}} if $\mathcal{W}^-_g = 0$.
Note that, by reversing orientation, a self-dual manifold is converted into
an anti-self-dual manifold, and vice versa. 
There are now so many known examples of anti-self-dual metrics on 
various compact four-manifolds, that it is difficult to give a complete 
list here, and we refer the reader to \cite{ViaclovskyIndex} for a 
recent list of references. 

  The deformation theory of anti-self-dual metrics is roughly analogous 
to the theory of deformation of complex structures. 
If $(M,g)$ is an anti-self-dual four-manifold, the anti-self-dual 
deformation complex is given by 
\begin{align}
\label{thecomplex}
\Gamma(T^*M) \overset{\mathcal{K}_g}{\longrightarrow} 
\Gamma(S^2_0(T^*M))  \overset{\mathcal{D}}{\longrightarrow}
\Gamma(S^2_0(\Lambda^2_+)),
\end{align}
where $\mathcal{K}_g$ is the conformal Killing operator defined 
by 
\begin{align}
( \mathcal{K}_g(\omega))_{ij} = \nabla_i \omega_j + \nabla_j \omega_i - 
\frac{1}{2} (\delta \omega) g, 
\end{align}
with $\delta \omega = \nabla^i \omega_i$,  
$S^2_0(T^*M)$ denotes traceless symmetric tensors, 
and $\mathcal{D} = (\mathcal{W}^+)_g'$ is the linearized self-dual Weyl curvature 
operator.

 If $M$ is a compact manifold then there is a formula for the index depending 
only upon topological quantities. The analytical index is given by
\begin{align}
Ind(M, g) = \dim( H^0(M,g)) -  \dim( H^1(M,g)) + \dim( H^2(M,g)),
\end{align}
where $H^i(M,g)$ is the $i$th cohomology of the complex \eqref{thecomplex}, 
for $i = 0,1,2$. 
The index is given in terms of topology via the Atiyah-Singer index theorem
\begin{align}
\label{manifoldindex}
Ind(M, g) = \frac{1}{2} ( 15 \chi(M) + 29 \tau(M)), 
\end{align}
where $\chi(M)$ is the Euler characteristic and 
$\tau(M)$ is the signature of $M$, see \cite{KotschickKing}.

The cohomology groups of the complex \eqref{thecomplex} yield information about 
the local structure of the moduli space of anti-self-dual conformal
classes, which we briefly recall \cite{Itoh2, KotschickKing}.
There is a map 
\begin{align}
\Psi: H^1(M,g) \rightarrow H^2(M,g)
\end{align} 
called the {\em{Kuranishi map}}
which is equivariant with respect to the action of $H^0$, 
and the moduli space of anti-self-dual conformal structures
near $g$ is localy isomorphic to 
$\Psi^{-1}(0) / H^0$. Therefore, if $H^2 = 0$, the moduli 
space is locally isomorphic to $H^1 / H^0$.

In this paper, we will be concerned with orbifolds
in dimension four with isolated singularities modeled on $\RR^4 / \Gamma$,
where $\Gamma$ is a finite subgroup of ${\rm{SO}}(4)$ 
acting freely on $\RR^4 \setminus \{0 \}$. 
We will say that $(M,g)$ is a {\em{Riemannian orbifold}} if $g$
is a smooth metric away from the singular points, and at any singular point, 
the metric is locally the quotient of a smooth $\Gamma$-invariant metric on $B^4$ under 
the orbifold group $\Gamma$.

The above results regarding the Kuranishi map are also valid 
for anti-self-dual Riemannian orbifolds. 
However, the index formula \eqref{manifoldindex} does not hold
without adding a correction term. 
In \cite{Kawasaki}, Kawasaki proved a version of the Atiyah-Singer index theorem  
for orbifolds, and gave a general formula for the correction term. 
Our first result is an explicit formula for this correction term 
for the complex \eqref{thecomplex} in the case that $\Gamma$ is an action 
of a cyclic group. In order to state this, we first make some definitions. 

For $1 \leq q < p$ relatively prime integers, we denote by $\Gamma_{(q,p)}$ the 
cyclic action 
\begin{align}
\label{qpaction}
\left(
\begin{matrix}
\exp^{2 \pi i k / p}   &  0 \\
0    & \exp^{2 \pi i k q / p } \\
\end{matrix}
\right),  \ \ 0 \leq k < p,
\end{align}
acting on $\RR^4$, which we identify with $\CC^2$ using
$z_1 = x_1 + i x_2$, and $z_2 = x_3 + i x_4$.
We will also refer to this action as a type $(q,p)$-action.
\begin{definition}{\em
A group action $\Gamma_1: G \rightarrow {\rm{SO}(4)}$ is {\em{conjugate}}
to another group action  $\Gamma_2: G \rightarrow {\rm{SO}(4)}$
if there exists an element $O \in O(4)$ 
such that for any $g \in G$, $\Gamma_1(g) \circ O = O \circ \Gamma_2(g)$. 
If $O \in {\rm{SO}}(4)$, then the actions are said to be 
orientation-preserving conjugate, while if $O \notin {\rm{SO}}(4)$,
the actions are said to be orientation-reversing conjugate.
}
\end{definition}

\begin{remark}
\label{actrem}
{\em
We note the important fact that if $\Gamma$ is an $\rm{SO}(4)$ 
representation of a cyclic group, then $\Gamma$ is orientation-preserving 
conjugate to a $\Gamma_{(q,p)}$-action \cite{MCC};
we therefore only need consider the $\Gamma_{(q,p)}$-actions.
Furthermore, for $1 \leq q , q' < p$, if a $\Gamma_{(q,p)}$-action is 
orientation-preserving conjugate to a $\Gamma_{(q',p)}$-action then 
$q q' \equiv 1 \mod p$.
We also note that a $\Gamma_{(q,p)}$-action is orientation-reversing conjugate 
to a $\Gamma_{(p-q,p)}$-action.
}
\end{remark}
We will employ the following modified Euclidean algorithm.
For $1 \leq q < p$ relatively prime integers, write
\begin{align}
\begin{split}\label{mea}
p&=e_1q-a_1\\
q&=e_2a_1-a_2\\
&\hspace{2mm} \vdots \\
a_{k-2}&=e_ka_{k-1}-1,
\end{split}
\end{align}
where $e_i \geq 2$, and $0 \leq a_i < a_{i-1}$. This can 
also be written as the continued fraction expansion
\begin{align}
\frac{p}{q} = \cfrac{1}{e_1 - \cfrac{1}{e_2 - \cdots \cfrac{1}{e_k}}}.
\end{align}
We refer to the integer $k$ as the {\em{length}} of the modified Euclidean algorithm. 

Our main theorem expresses the correction term in the index theorem 
in terms of the $e_i$ and the length of the modified Euclidean algorithm: 
\begin{theorem}
\label{mainthm}
Let $(M,g)$ be a compact anti-self-dual orbifold with a single 
orbifold point of type $(q,p)$. 
The index of the anti-self-dual deformation complex on $(M,g)$ is given by 
\begin{align}
Ind(M,g)=
\begin{cases}
 \displaystyle \frac{1}{2}(15\chi_{top}+29\tau_{top})+\sum_{i=1}^{k}4e_i-12k-2  &\text{   when $q \neq p-1$}\\
 \displaystyle \frac{1}{2}(15\chi_{top}+29\tau_{top}) - 4p + 4
&\text{   when $q=p-1$}.
\end{cases}
\end{align}
\end{theorem}
In some other special cases, the correction term may be written directly in 
terms of $p$. For example, if $q = 1$, and $p > 2$, we have 
\begin{align}
\sum_{i=1}^{k}4e_i-12k-2 = 4p -14.
\end{align}
We note that the cases $q = 1$ and $q = p-1$ 
were proved earlier in \cite{ViaclovskyIndex} using a different method. 

\begin{remark}
\label{srmk}
{\em
While Theorem \ref{mainthm} is stated in the case of a single 
orbifold point for simplicity, if a compact anti-self-dual 
orbifold has several cyclic quotient 
orbifold points, then a similar formula holds, with 
the correction term simply being the sum of the corresponding 
correction terms for each type of orbifold point. }
\end{remark}

\subsection{Asymptotically locally Euclidean spaces}
 Many interesting examples of anti-self-dual metrics are complete and non-compact. 
Given a compact Riemannian orbifold $(\hat{X}, \hat{g})$ with 
non-negative scalar curvature, letting $G_p$ denote the Green's function 
for the conformal Laplacian associated with any point $p$, 
the non-compact space $X = \hat{X} \setminus \{p\}$
with metric $g_p = G_p^{2}\hat{g}$
is a complete scalar-flat orbifold. 
Inverted normal coordinates in the metric $\hat{g}$ in a neighborhood
of the point $p$, give rise to a coordinate system 
in a neighborhood of infinity of $X$, which motivates the 
following:
\begin{definition}
\label{ALEdef}
{\em
 A noncompact Riemannian orbifold $(X^4,g)$ 
is called {\em{asymptotically locally 
Euclidean}} or {\em{ALE}} of order $\tau$ if 
there exists a finite subgroup 
$\Gamma \subset {\rm{SO}}(4)$ 
acting freely on $\RR^4 \setminus \{0\}$, 
and a diffeomorphism 
$\phi : X \setminus K \rightarrow ( \RR^4 \setminus B(0,R)) / \Gamma$ 
where $K$ is a compact subset of $X$, satisfying 
$(\phi_* g)_{ij} = \delta_{ij} + O( r^{-\tau})$ and 
$\partial^{|k|} (\phi_*g)_{ij} = O(r^{-\tau - k })$
for any partial derivative of order $k$, as
$r \rightarrow \infty$, where $r$ is the distance to some fixed basepoint.  
}
\end{definition}

 An {\em{orbifold compactification}} of an ALE space $(X,g)$,
is a choice of a conformal factor $u : X \rightarrow \RR_+$ 
such that $u = O(r^{-2})$ as $r \rightarrow \infty$. 
The space $(X, u^2 g)$ then compactifies to a
$C^{1,\alpha}$ orbifold. If $(X,g)$ is anti-self-dual, then there
moreover exists a $C^{\infty}$-orbifold conformal compactification 
$(\hat{X}, \hat{g})$ \cite[Proposition 12]{CLW}. 

\begin{remark}{\em
It is crucial to note that if $(X,g)$ is an anti-self-dual 
ALE space with a $\Gamma$-action at 
infinity, then the conformal compactification $(\hat{X}, \hat{g})$ with the 
anti-self-dual orientation has a  
$\tilde{\Gamma}$-action at the orbifold point where $\tilde{\Gamma}$ is 
orientation-reversing conjugate to $\Gamma$. In the case of a cyclic 
group, if the action at infinity of the anti-self-dual ALE space $(X,g)$ is of type 
$(q,p)$, then the action at the orbifold point of the compactification 
$(\hat{X}, \hat{g})$ with the anti-self-dual orientation is of type $(p-q,p)$.
}
\end{remark}

 Many examples of anti-self-dual ALE spaces with nontrivial group 
at infinity have been discovered. The first non-trivial example was
due to Eguchi and Hanson, who found
a Ricci-flat anti-self-dual metric on $\mathcal{O}(-2)$ which is ALE with 
group $\ZZ / 2 \ZZ$ at infinity \cite{EguchiHanson}. 
Gibbons-Hawking then wrote down an metric ansatz depending on the choice of 
$n$ monopole points in $\RR^3$, giving an anti-self-dual ALE hyperk\"ahler metric 
with cyclic action at infinity contained in ${\rm{SU(2)}}$, which 
are called multi-Eguchi-Hanson metrics \cite{GibbonsHawking, Hitchin2}. 

Using the Joyce construction from \cite{Joyce1995}, 
Calderbank and Singer produced many examples of toric 
ALE anti-self-dual metrics, which are moreover 
scalar-flat K\"ahler, and have cyclic groups at 
infinity contained in ${\rm{U}}(2)$ \cite{CalderbankSinger}.
For a type $(q,p)$-action, the space $X$ is the minimal Hirzebruch-Jung 
resolution 
of $\CC^2/ \Gamma_{(q,p)}$, with exceptional divisor 
given by the union of $2$-spheres $S_1 \cup \cdots \cup S_k$, with intersection
matrix 
\begin{align}
\label{intersect}
(S_i \cdot S_j) = 
\left(
\begin{matrix}
-e_1  & 1 & 0 & \cdots & 0 \\
1    & - e_2 & 1 & \cdots & 0 \\
0 & 1 & - e_3 & \cdots & 0 \\
\vdots & \vdots & \vdots &  & \vdots \\
0 & 0 & 0 & \cdots & - e_k \\
\end{matrix}
\right),
\end{align}
where the $e_i$ and $k$ are defined above in \eqref{mea}
with $e_i \geq 2$. 
  The K\"ahler scalar-flat metric on $X$ is then written down 
explicitly using the Joyce ansatz from \cite{Joyce1995}. 
We do not require the details of the construction here, 
but only note the following: 
For $1 < q$ the identity component of the isometry 
group of these metrics is a real $2$-torus,
and for $q = 1$, it is ${\rm{U}}(2)$. 

When $q = p - 1$, these metrics 
are the {\em{toric}} Gibbons-Hawking multi-Eguchi-Hanson metrics 
(when all monopole points are on a common line). 
In this case $k = p-1$ and $e_i = 2$ for $1 \leq i \leq k$. 
The moduli space of toric metrics in this case is of dimension 
$p-2$. But the moduli space of all multi-Eguchi-Hanson 
metrics is of dimension $3(p-2)$. So it is well-known that 
these metrics admit non-toric anti-self-dual deformations.
When $q = 1$, these metrics agree with the 
LeBrun negative mass metrics on $\mathcal{O}(-p)$ discovered 
in \cite{LeBrunnegative}. In this case $k =1$ and $e_1 = p$. 
For $p > 2$, it was recently shown in \cite{HondaOn, ViaclovskyIndex} 
that these spaces also admit non-toric anti-self-dual 
deformations. 
Theorem \ref{CScor} will give a vast 
generalization of this phenomenon to the general case $1 < q < p-1$.
The proof of Theorem \ref{CScor} relies on the following explicit formula for the 
index of the complex \eqref{thecomplex} on the 
conformal compactification of these metrics:
\begin{theorem}
\label{CSthm}
Let $(\hat{X}, \hat{g})$ be the orbifold conformal compactification 
of a Calderbank-Singer space $(X,g)$ with a $(q,p)$-action at infinity.
Then the index of of anti-self-dual deformation complex is given by 
\begin{align}
Ind(\hat{X}, \hat{g}) =
\begin{cases}
\displaystyle 5k+5-\sum_{i=1}^{k}4e_i &\text{   when $q \neq 1$}\\
\displaystyle-4p+12 &\text{   when $q=1$},
\end{cases}
\end{align}
where the integers $k$ and $e_i$, $1 \leq i \leq k$, 
are the integers occuring in  
the modified Euclidean algorithm defined in \eqref{mea}. 
\end{theorem}
We note that if $q = p - 1$, the index simplifies to $-3 p + 8$. 
A consequence of the above is that the Calderbank-Singer spaces 
admit large families of non-toric anti-self-dual deformations, 
thereby yielding many new examples:

\begin{theorem}
\label{CScor} 
Let $(X,g)$ be a Calderbank-Singer space with a $(q,p)$-action at infinity,
and $(\hat{X}, \hat{g})$ be the orbifold conformal compactification.
Let $\mathcal{M}_{\hat{g}}$ denote the moduli space of anti-self-dual 
conformal structures near $(\hat{X}, \hat{g})$. Then,
\begin{itemize}
\item
If $q = 1$ and $p =2$, then $\hat{g}$ is rigid. 
\item
If $q = 1$ and $p = 3$, then $\mathcal{M}_{\hat{g}}$ is locally of dimension $1$.
\item
If $q = 1$ and $p > 3$, then $\mathcal{M}_{\hat{g}}$ is locally of dimension $4p -12$.
\item
If $q = p -1$, then  $\mathcal{M}_{\hat{g}}$ is locally of dimension $3p -7$.
\item
If $1 < q < p - 1$, then  $\mathcal{M}_{\hat{g}}$ is locally of dimension at least 
\begin{align}
\label{csest}
\dim(H^1) - 2 = -5k-5+\sum_{i=1}^{k}4e_i. 
\end{align}
\end{itemize}
Consequently, if $p  > 2$, these spaces admit non-toric anti-self-dual deformations.
\end{theorem}
\begin{remark}{\em
Theorem \ref{CSthm} 
could be equivalently stated in terms of the ALE metrics rather than the 
compactified metrics. However, the definition of the index on an ALE
space involves defining certain weighted spaces; 
see \cite[Proposition 3.1]{ViaclovskyIndex} for the precise formula which relates
the index on the ALE space to the index on the compactification; for our 
purposes here, we only require the statement on the compactification.
Similarly, Theorem~\ref{CScor} could be equivalently stated in terms
of anti-self-dual ALE deformations of the ALE model. 
}
\end{remark} 
By a result of LeBrun-Maskit, $H^2(\hat{X}, \hat{g}) = 0$ for these metrics, 
so the actual moduli space is locally isomorphic to $H^1/ H^0$,  
\cite[Theorem 4.2]{LeBrunMaskit}.
Therefore the moduli space
could be of dimension $\dim(H^1)$, $\dim(H^1) - 1$, or $\dim(H^1) - 2$.
This action in the toric multi-Eguchi-Hanson case $q = p -1$ is well-known; 
in this case for $p \geq 3$, $\dim(H^1) = 3p - 6$, and 
the dimension of the moduli space is equal to $\dim(H^1) -1 = 3p - 7$. 
In the LeBrun negative mass case $q = 1$, this 
action was recently completely determined by Nobuhiro Honda using
arguments from twistor theory \cite{HondaOn}. 
For $ 1 < q < p-1$, further arguments are needed to determine this action
explicitly; this is an interesting problem.

\subsection{Weighted projective spaces} 
We first recall the definition of weighted projective spaces in real dimension four:
\begin{definition} {\em{ For relatively prime integers $1 \leq r \leq q \leq p$, 
the {\em{weighted projective space}} $\mathbb{CP}^2_{(r,q,p)}$ is $S^{5}/\mathbb{C}^*$, 
where $\mathbb{C}^*$ acts by
\begin{align}
(z_0,z_1,z_2)\mapsto (e^{ir\theta}z_0,e^{iq\theta}z_1 ,e^{ip\theta}z_2),
\end{align}
for $0\leq \theta <2\pi$. 
}}
\end{definition}

The space $\mathbb{CP}^2_{(r,q,p)}$ has the structure of a compact complex orbifold.
In \cite{Bryant}, 
Bryant proved that every weighted projective space admits a Bochner-K\"ahler metric.  
Subsequently, David and Gauduchon gave a simple and direct construction of 
these metrics \cite{DavidGauduchon}. Using an argument due to Apostolov, they
also showed that this metric is the unique Bocher-K\"ahler metric on a given weighted 
projective space \cite[Appendix D]{DavidGauduchon},
and thus we will call this metric the {\em{canonical}} Bochner-K\"ahler metric.  
In complex dimension two, the Bochner tensor is the same as the 
anti-self-dual part of the Weyl 
tensor so Bochner-K\"ahler metrics are the same as self-dual K\"ahler metrics.  

The work of Derdzinski \cite{Derdzinski} showed that a 
self-dual K\"ahler metric $g$ is conformal to a 
self-dual Hermitian Einstein metric on 
$M^*:=\{p\in M : R(p)\neq 0\}$, given by $\tilde{g}=R^{-2}g$, where
$R$ is the scalar curvature.  This conformal metric is not K\"ahler unless 
$R$ is constant.  Conversely, Apostolov and Gauduchon \cite{ApostolovGauduchon} showed 
that any self-dual Hermitian Einstein 
metric that is not conformally flat is of the form $\tilde{g}$ for a unique 
self-dual K\"ahler metric $g$ with $R\neq 0$.  

For a weighted projective space $\mathbb{CP}^2_{(r,q,p)}$, there are the following 3 cases: 
\begin{itemize}
\item
When $p < r+q$ the canonical Bochner-K\"ahler metric has $R>0$ everywhere, so it is conformal to a Hermitian Einstein metric with positive Einstein constant.  
\item
When $p = r+q$ the canonical Bochner-K\"ahler metric has $R>0$ except at one point, so it is conformal to a complete Hermitian Einstein metric with vanishing Einstein constant outside this point.  

\item When $p > r+q$ the canonical Bochner-K\"ahler metric has $R$ vanishing along a hypersurface and the complement is composed of two open sets on which the metric is conformal to a Hermitian Einstein metric with negative Einstein constant.  
\end{itemize}
For $x \in \RR$, $\lfloor x \rfloor$ denotes the integer part of $x$,
and $\{x\} = x -\lfloor x \rfloor $ denotes the fractional part of $x$. 
We also define the integer $\epsilon$ by 
\begin{align}
\epsilon = 
\begin{cases}
0  & \text{if } p \not\equiv q \text{ mod } r \mbox{ and } p \not\equiv r \text{ mod } q\\
1 & \text{if } p \equiv q \text{ mod } r \mbox{ or } p \equiv r \text{ mod } q,
\mbox{ but not both}, \\
2 &  \text{if } p \equiv q \text{ mod } r \mbox{ and } p \equiv r \text{ mod } q. 
\\
\end{cases}
\end{align}
Our main result for the index on weighted projective spaces 
is the following, with the answer depending upon certain number-theoretic 
properties of the triple $(r,q,p)$: 
\begin{theorem}
\label{introthm}
Let $g$ be the canonical Bochner-K\"ahler metric with reversed 
orientation on $\overline{\mathbb{CP}}^2_{(r,q,p)}$.
Assume that $1<r<q<p$. If $r+q\geq p$ then 
\begin{align}
Ind(\overline{\mathbb{CP}}^2_{(r,q,p)},g)=2.
\end{align}
If $r+q<p$, then
\begin{align}
Ind(\overline{\mathbb{CP}}^2_{(r,q,p)},g)=
\begin{cases}
2 +2\epsilon -4\lfloor \frac{p}{qr} \rfloor &\text{   when $\{ \frac{p}{qr}\}<\{\frac{q^{-1;r}p}{r}\}$}\\
-2 +2\epsilon -4\lfloor \frac{p}{qr} \rfloor &\text{   when $\{ \frac{p}{qr}\}>\{\frac{q^{-1;r}p}{r}\}$}.
\end{cases}
\end{align}
\end{theorem}
We note that in the case $\{ \frac{p}{qr}\}<\{\frac{q^{-1;r}p}{r}\}$,
the integer $\epsilon$ can  only be $0$ or $1$; the integer $2$ 
does not actually occur in this case. Thus there are exactly $5$ 
cases which do in fact all occur, see Section \ref{wpssec}. 

Theorem \ref{introthm} implies the following result regarding 
the moduli space of anti-self-dual metrics on $\overline{\mathbb{CP}}^2_{(r,q,p)}$:
\begin{theorem}
\label{wpsthm} 
Let $g$ be the canonical Bochner-K\"ahler metric with reversed 
orientation on $\overline{\mathbb{CP}}^2_{(r,q,p)}$.
Assume that $1<r<q<p$. Then, 
\begin{itemize}
\item
If $p \leq q + r$ then $[g]$ is isolated as an anti-self-dual conformal
class.
\item
If $p > q + r$ then the moduli space of anti-self-dual
orbifold conformal classes near $g$, $\mathcal{M}_g$, 
is of dimension at least
\begin{align}
\dim( \mathcal{M}_g) \geq 
\begin{cases}
4\lfloor \frac{p}{qr} \rfloor- 2 -2\epsilon  &\text{   when $\{ \frac{p}{qr}\}<\{\frac{q^{-1;r}p}{r}\}$}\\
4\lfloor \frac{p}{qr} \rfloor +2 -2\epsilon &\text{   when $\{ \frac{p}{qr}\}>\{\frac{q^{-1;r}p}{r}\}$}.
\end{cases}
\end{align}
\end{itemize}
\end{theorem}

\begin{remark} {\em
Since the case $p < q + r$ is conformal to an Einstein 
metric, it is perhaps not surprising (although not obvious) that these metrics 
are also isolated as self-dual metrics. But the non-trivial anti-self-dual deformations
we have found in the case $p > q +r$ are quite surprising, since these
metrics are rigid as Bochner-K\"ahler metrics.
}
\end{remark}

The proof of Theorem \ref{wpsthm} also relies on 
the fact that $H^2(M,g) = 0$ for these metrics, see Corollary \ref{h2wps}
below. Then as pointed out above, the actual
moduli space is locally isomorphic to $H^1/ H^0$, so the moduli space
could be of dimension $\dim(H^1)$, $\dim(H^1) - 1$, or $\dim(H^1) - 2$.
As in the case of the Calderbank-Singer spaces, 
we do not determine this action explicitly here; this is another very 
interesting problem. 

\subsection{Outline of paper} We begin in Section \ref{ogi} by
recalling Kawasaki's orbifold index theorem, and apply it to
the complex \eqref{thecomplex}. Then in Section \ref{gcyclic}, 
we analyze the correction terms for cyclic group actions, 
culminating in the following formula for the index in terms of 
the following trigonometric sum when $1 < q < p -1$: 
\begin{align}
\begin{split}
\label{N(q,p)intro}
Ind_{\Gamma}(\hat{M})&=\frac{1}{2}(15\chi_{top}+29\tau_{top})-6+\frac{14}{p}\sum_{j=1}^{p-1}\Big[\cot(\frac{\pi}{p}j)\cot(\frac{\pi}{p}qj)\Big]\\
&\phantom{==}-\frac{2}{p}\sum_{j=1}^{p-1}\Big[\cot(\frac{\pi}{p}j)\cot(\frac{\pi}{p}qj)\cos(\frac{2\pi}{p}j)\cos(\frac{2\pi}{p}qj)\Big].
\end{split}
\end{align}
We note that the quantity 
\begin{align}
\label{Dedekind}
s(q,p) = \frac{1}{4p}\sum_{j=1}^{p-1}\Big[\cot(\frac{\pi}{p}j)\cot(\frac{\pi}{p}qj)\Big]
\end{align}
is the well-known Dedekind sum \cite{Rademacher2}. This has 
a closed form expression in several special cases, but not in 
general. It is not surprising that this term appears, since 
Dedekind sums arise naturally in the index theorem 
for the signature complex \cite{HirzebruchZagier, Katase1, Zagier}. 
However, for the
anti-self-dual deformation complex, the interaction of the Dedekind sum term 
with the final term in \eqref{N(q,p)intro} makes a huge difference. 
In particular, formula  \eqref{N(q,p)intro} says that 
the sum of these terms must always be an integer!

For $x \in \RR \setminus \ZZ$, we define the sawtooth function $((x)) = \{ x \} - \frac{1}{2}$. 
In Section \ref{nontop}, we show that when $1 < q < p-1$, 
the non-topological terms in \eqref{N(q,p)intro}
can be rewritten as a Dedekind sum plus terms involving the sawtooth function:
\begin{align} 
N(q,p)=-6+\frac{12}{p}\sum_{j=1}^{p-1}\cot(\frac{\pi}{p}j)\cot(\frac{\pi}{p}qj)-4\bigg(\bigg(\frac{q^{-1;p}}{p}\bigg)\bigg) -4\bigg(\bigg(\frac{q}{p}\bigg)\bigg),
\end{align}
where $q^{-1;p}$ is the inverse of $q$ modulo $p$. 
In Section \ref{explicit} we use this, together with classical reciprocity for 
Dedekind sums to prove Theorem \ref{mainthm}.
The results dealing with the Calderbank-Singer spaces, Theorems \ref{CSthm} and \ref{CScor}, are proved in Section~\ref{CSindex}. 
Finally, in Section \ref{wpssec}, we present a complete analysis of the index 
for the canonical Bochner-K\"ahler metric on a weighted projective space, and 
prove Theorem \ref{wpsthm}. Interestingly, an important ingredient is 
Rademacher's triple reciprocity formula for Dedekind sums \cite{Rademacher1}. 
We conclude the paper with some remarks on the 
number-theoretic condition on the triple $(r,q,p)$ which occurs in 
Theorem \ref{introthm}.

\subsection{Acknowledgements}
The authors would like to thank John Lott for crucial discussions about index theory.
Thanks are also due to Nobuhiro Honda for many valuable 
discussions on the moduli space of anti-self-dual metrics.


\section{The orbifold $\Gamma$-Index}
\label{ogi}
For an orbifold $(M,g)$, the $\Gamma$-Index is given analytically by
\begin{align}
Ind_{\Gamma}(M,g) = \dim (H^0(M,g)) - \dim (H^1(M,g)) + \dim (H^2(M,g)).
\end{align}
From Kawasaki's orbifold index theorem \cite{Kawasaki}, it follows that we have a $\Gamma$-index formula of the form
\begin{align}
\label{I-G}
Ind_{\Gamma}(M)=\frac{1}{2}(15\chi_{orb}(M)+29\tau_{orb}(M))+\frac{1}{|\Gamma|}\sum_{\gamma \neq Id}\frac{ch_{\gamma}(i^*\sigma)}{ch_{\gamma}({\lambda_{-1}N_{\mathbb{C}}})}.
\end{align}
where the quantity $\chi_{orb}(M)$ is the orbifold Euler characteristic defined by
\begin{align}
\chi_{orb}(M)=\frac{1}{8\pi^2}\int_{M}(|W|^2-\frac{1}{2}|Ric|^2+\frac{1}{6}R^2)dV_{g},
\end{align}
the quantity $\tau_{orb}(M)$ is the orbifold signature defined by
\begin{align}
\tau_{orb}(M)=\frac{1}{12\pi^2}\int_{M}(|W^+|^2-|W^-|^2)dV_{g},
\end{align}
and the quantity $\frac{ch_{\gamma}(i^*\sigma)}{ch_{\gamma}(\lambda_{-1}N_{\mathbb{C}})}$ is a correction term depending upon the action of $\gamma$ on certain bundles, which we will describe in what follows.  

In the next subsection, we compute the trace of the action of $\gamma$, an element in the orbifold group $\Gamma$, on the bundles $[N_{\mathbb{C}}]$, $[S^2_0(N_{\mathbb{C}})]$ and $[S^2_0(\Lambda^2_+)]$ over the fixed point set, which we then use to compute a general formula for the $\frac{ch_{\gamma}(i^*\sigma)}{ch_{\gamma}(\lambda_{-1}N_{\mathbb{C}})}$ term.  Then we give the orbifold Euler characteristic and orbifold signature in terms of the topological Euler characteristic and topological signature and correction terms also depending upon the $\gamma$-action respectively.  Finally, we combine this information into a formula for the orbifold $\Gamma$-Index.

\subsection{Group action on bundles}
In order to compute the $\Gamma$-Index, we first need to find the trace of the $\gamma$-action, for every $\gamma$ in $\Gamma$, on the pullback of the complexified principal symbol, $i^*\sigma$, where
\begin{align}
\label{pr1}
i:p\rightarrow M
\end{align}
is the inclusion map from the fixed point $p$ into the orbifold $M$.  In this case
\begin{align}
\label{symbol}
i^*\sigma=[N_{\mathbb{C}}]-[S^2_0(N_{\mathbb{C}})]+[S^2_0\Lambda^2_+].
\end{align}
For a general $\gamma$ of the form
\begin{align}
\gamma = \left( {\begin{array}{*{20}c}
   \cos\theta_1 & -\sin\theta_1 & 0 & 0 \\
  \sin\theta_1 & \cos\theta_1 & 0 & 0 \\
   0 & 0 & \cos\theta_2 & -\sin\theta_2 \\
    0 & 0 & \sin\theta_2 & \cos\theta_2
      \end{array} } \right),
\end{align}
fixing the point $p$, the normal bundle is trivial, so $N_\mathbb{C}:=N\otimes \mathbb{C}=\mathbb{C}^4$, and we have the following proposition.
\begin{proposition}
\label{traceprop}
The trace of the $\gamma$-action on the components of $i^*\sigma$ is as follows:
\begin{enumerate}
\item $\displaystyle tr(\gamma|_{N_{\mathbb{C}}})=2\cos (\theta_1)+2\cos(\theta_2)$,
\item $\displaystyle tr(\gamma|_{S^2_0(N_{\mathbb{C}})})=1+2\cos(\theta_1+\theta_2)+ 2\cos(-\theta_1+\theta_2)+4\cos(\theta_1+\theta_2)\cos(-\theta_1+\theta_2)$,
\item $\displaystyle tr(\gamma|_{S^2_0(\Lambda^2_+)})=2cos(\theta_1+\theta_2)+4cos^2(\theta_1+\theta_2)-1.$
\end{enumerate}
\end{proposition}
\begin{proof}
The normal bundle can be written as $N=x_1\oplus \cdot \cdot \cdot \oplus x_4$ in real coordinates.  After complexifying the normal bundle we can diagonalize $\gamma$ to write
\begin{align}
\gamma|_{N_{\mathbb{C}}} = 
\left( {\begin{array}{*{20}c}
   e^{i\theta_1} & 0 & 0 & 0 \\
   0 & e^{-i\theta_1} & 0 & 0  \\
   0 & 0 & e^{i\theta_2}& 0  \\
    0 & 0 & 0 & e^{-i\theta_2}
      \end{array} } \right),
\end{align}
with respect to the complex basis $\{\lambda_1\oplus \lambda_2\oplus \lambda_3 \oplus \lambda_4\}=\mathbb{C}^4$, where
\begin{align}
\{2x_1, 2x_2, 2x_3, 2x_4 \} = \{ 
\lambda_1-i\lambda_2, i\lambda_1-\lambda_2,
\lambda_3-i\lambda_4,i\lambda_3-\lambda_4  \}.
\end{align}
Formula $(1)$ follows immediately. 

Next, to see how $\gamma$ acts on $S^2_0(N_g)=\Lambda^2_+ \otimes \Lambda^2_-$ we first examine how $\gamma$ acts on $\Lambda^2_+$ and $\Lambda^2_-$ independently.  
We use the following basis for $\Lambda^2_+$:
\begin{align}
\begin{split}
\omega_1^+&=\frac{1}{2} [d\lambda_2\wedge d\lambda_1+d\lambda_4\wedge d\lambda_3],\\
\omega_2^+&=
\frac{1}{2}[d\lambda_1\wedge d\lambda_3+d\lambda_4\wedge d\lambda_2],\\
\omega_3^+
&=\frac{1}{2}[id\lambda_1\wedge d\lambda_3+id\lambda_2\wedge d\lambda_4],
\end{split}
\end{align}
and the following basis for $\Lambda^2_-$:
\begin{align}
\begin{split}
\omega_1^-&=\frac{1}{2} [d\lambda_2\wedge d\lambda_1-d\lambda_4\wedge d\lambda_3],\\
\omega_2^-&=\frac{1}{2}[id\lambda_3\wedge d\lambda_2+id\lambda_4\wedge d\lambda_1],\\
\omega_3^-&=\frac{1}{2}[d\lambda_2\wedge d\lambda_3+d\lambda_4\wedge d\lambda_1].
\end{split}
\end{align}
So we see that $\gamma$ acts on $\Lambda_+^2$ by
\begin{align}
\begin{split}
\gamma(\omega_1^+)&=\omega_1^+\\
\gamma(\omega_2^+)
&=\frac{1}{2}[e^{i(\theta_1+\theta_2)}(\omega_2^+-i\omega_3^+)+e^{-i(\theta_1+\theta_2)}(\omega_2^++i\omega_3^+)]\\
\gamma(\omega_3^+)
&=\frac{1}{2}[e^{i(\theta_1+\theta_2)}(\omega_3^++i\omega_2^+)+e^{-i(\theta_1+\theta_2)}(\omega_3^+-i\omega_2^+)],
\end{split}
\end{align}
and $\gamma$ acts on $\Lambda_-^2$ by
\begin{align}
\begin{split}
\gamma(\omega_1^-)&=\omega_1^-\\
\gamma(\omega_2^-)
&=\frac{1}{2}[e^{i(-\theta_1+\theta_2)} (\omega_2^--i\omega_3^-)+e^{i(\theta_1-\theta_2)}(\omega_2^-+i\omega_3^-)]\\
\gamma(\omega_3^+)
&=\frac{1}{2}[e^{i(-\theta_1+\theta_2)}(\omega_3^-+i\omega_2^-)+e^{i(\theta_1-\theta_2)}(\omega_3^--i\omega_2^-)].
\end{split}
\end{align}
Therefore, we can write
\begin{align}
 \gamma|_{\Lambda_+^2}
= \left(
\begin{matrix}
1&0&0\\
0& \cos(\theta_1+\theta_2) &- \sin(\theta_1+\theta_2) \\
   0& \sin(\theta_1+\theta_2) & \cos(\theta_1+\theta_2)
\end{matrix} 
\right),
\end{align}      
and
\begin{align}
\gamma|_{\Lambda_-^2}
= \left(
\begin{matrix}
1&0&0\\
0& \cos(-\theta_1+\theta_2) &- \sin(-\theta_1+\theta_2) \\
 0& \sin(-\theta_1+\theta_2) & \cos(-\theta_1+\theta_2)
\end{matrix} \right).
\end{align}      
To derive $(2)$, we compute
\begin{align}
\begin{split}
tr(\gamma|_{S^2_0 N_{\mathbb{C}}})&=tr(\gamma|_{\Lambda^2_+\otimes \Lambda^2_-})\\
&=tr(\gamma|_{\Lambda^2_+})\cdot tr(\gamma|_{\Lambda^2_-})\\
&=(1+ 2\cos(\theta_1+\theta_2))\cdot (1+ 2\cos(-\theta_1+\theta_2))\\
=1+2\cos(\theta_1+\theta_2)&+ 2\cos(-\theta_1+\theta_2)+4\cos(\theta_1+\theta_2)\cos(-\theta_1+\theta_2).
\end{split}
\end{align}
Next, to see how $\gamma$ acts on $S^2_0(\Lambda^2_+)$, decompose
\begin{align}
S^2_0\Lambda_+^2=[\mathbb{C}\otimes (\omega_2^+\oplus \omega_3^+)]\oplus S^2_0(\omega_2^+\oplus \omega_3^+)\oplus tr,
\end{align}
where $ tr=2\omega_1^+-(\omega_2^++\omega_3^+)$
denotes the trace component,
and write the basis of $S^2_0(\omega_2^+\oplus \omega_3^+)$ as
\begin{align}
 \{\omega_2^+\otimes \omega_2^+-\omega_3^+\otimes \omega_3^+, \omega_2^+\otimes \omega_3^++\omega_3^+\otimes \omega_2^+\}.
 \end{align}
We see that
\begin{align}
&\gamma|_{\omega_1^+\otimes (\omega_2^+ \oplus \omega_3^+)}=\left(
\begin{matrix}
	\cos(\theta_1+\theta_2) &- \sin(\theta_1+\theta_2) \\
  	\sin(\theta_1+\theta_2) & \cos(\theta_1+\theta_2)
\end{matrix} \right), \\
&\gamma|_{S^2_0(\omega_2^+\oplus \omega_3^2)}=\left( 
\begin{matrix}
	\cos^2(\theta_1+\theta_2) -\sin^2(\theta_1+\theta_2) & -2\sin(\theta_1+\theta_2) \cos(\theta_1+\theta_2) \\
  	2\sin(\theta_1+\theta_2) \cos(\theta_1+\theta_2) & \cos^2(\theta_1+\theta_2) -\sin^2(\theta_1+\theta_2)
\end{matrix} \right),\\
&\gamma|_{tr\in S^2_0\Lambda^2_+}=1.
\end{align}
Using these, we derive $(3)$ by computing
\begin{align}
\begin{split}
tr(\gamma|_{S^2_0\Lambda^2_+})&=[2cos(\theta_1+\theta_2)]+[4cos^2(\theta_1+\theta_2)-2]+[1]\\
&=2cos(\theta_1+\theta_2)+4cos^2(\theta_1+\theta_2)-1.
\end{split}
\end{align}
\end{proof}

\subsection{Equivariant Chern character}  
We next compute the term $\frac{ch_{\gamma}(i^*\sigma)}{ch_{\gamma}(\lambda_{-1}N_{\mathbb{C}})}$.
The numerator of this term is the ${\gamma}$-equivariant Chern character of the pullback of the principal symbol, $i^*\sigma$, described in $\eqref{pr1}$ and $\eqref{symbol}$.  The denominator is the ${\gamma}$-equivariant Chern character of the $K$-theoretic Thom class of the complexified normal bundle.  Since the normal bundle is trivial over the fixed point, this is
\begin{align}
\lambda_{-1}N_\mathbb{C}=[\Lambda^0(\mathbb{C}^4)]-[\Lambda^1(\mathbb{C}^4)]+[\Lambda^2(\mathbb{C}^4)]-[\Lambda^3(\mathbb{C}^4)]+[\Lambda^4(\mathbb{C}^4)].
\end{align}
Since the ${\gamma}$-equivariant Chern character is just the ${\gamma}$-action times 
the Chern character of each eigenspace, using Proposition \ref{traceprop}, we compute 
\begin{align}
\begin{split}
ch_{\gamma}(i^*\sigma)&=tr({\gamma}|_{N_\mathbb{C}})-tr({\gamma}|_{S^2_0N_\mathbb{C}})+tr({\gamma}|_{S^2_0\Lambda^2_+})\\
&=[2\cos(\theta_1)+2\cos(\theta_2)]\\
&\phantom{=}-[1+2\cos(\theta_1+\theta_2)+2\cos(-\theta_1+\theta_2)+4\cos(\theta_1+\theta_2)\cos(-\theta_1+\theta_2)]\\
&\phantom{=}+[2\cos(\theta_1+\theta_2)+4\cos^2(\theta_1+\theta_2)-1]\\
&=[2\cos\theta_1+2\cos\theta_2-2-2\cos(\theta_1)\cos(\theta_2)]\\
&\phantom{=}+[-2\sin(\theta_1)\sin(\theta_2)-8\cos(\theta_1)\cos(\theta_2)\sin(\theta_1)\sin(\theta_2)+8\sin^2(\theta_1)\sin^2(\theta_2)]\\
&=[-2(\cos\theta_1-1)(\cos\theta_2-1)]+[8(1-\cos^2\theta_1)(1-\cos^2\theta_2)]\\
&\phantom{=}+[-2\sin(\theta_1)\sin(\theta_2)-8\cos(\theta_1)\cos(\theta_2)\sin(\theta_1)\sin(\theta_2)].
\end{split}
\end{align}
Similarly, we compute
\begin{align}
\begin{split}
ch_{\gamma}(\lambda_{-1}N_{\mathbb{C}})&=tr({\gamma}|_{[\Lambda^0(\mathbb{C}^4)]})-tr({\gamma}|_{[\Lambda^1(\mathbb{C}^4)]})+tr({\gamma}|_{[\Lambda^2(\mathbb{C}^4)]})\\
&\phantom{=}-tr({\gamma}|_{[\Lambda^3(\mathbb{C}^4)]})+tr({\gamma}|_{[\Lambda^4(\mathbb{C}^4)]})
=4(\cos\theta_1-1)(\cos\theta_2-1).
\end{split}
\end{align}
Therefore
\begin{align}
\begin{split}
\frac{ch_{\gamma}(i^*\sigma)}{ch_{\gamma}(\lambda_{-1}N_{\mathbb{C}})}&=\Big[-\frac{1}{2}+2(1+\cos\theta_1)(1+\cos\theta_2)\Big]\\
&\phantom{=}-\bigg[\frac{2\sin(\theta_1)\sin(\theta_2)+8\cos(\theta_1)\cos(\theta_2)\sin(\theta_1)\sin(\theta_2)}{4(\cos\theta_1-1)(\cos\theta_2-1)}\bigg].
\end{split}
\end{align}
Since $\frac{\sin(\theta_1)\sin(\theta_2)}{(\cos\theta_1-1)(\cos\theta_2-1)}=\cot(\frac{\theta_1}{2})\cot(\frac{\theta_2}{2})$, we see that
\begin{align}
\begin{split}
\label{ch_g}
\frac{ch_{\gamma}(i^*\sigma)}{ch_{\gamma}(\lambda_{-1}N_{\mathbb{C}})}
&=-\frac{1}{2}+2(1+\cos\theta_1)(1+\cos\theta_2)-\frac{1}{2}\cot (\frac{\theta_1}{2})\cot(\frac{\theta_2}{2})\\
&\phantom{==}-2\cot(\frac{\theta_1}{2})\cot(\frac{\theta_2}{2})\cos(\theta_1)\cos(\theta_2).
\end{split}
\end{align}

\subsection{The $\Gamma$-Index}
For an orbifold with a single isolated singularity, 
we have a formula for the Euler characteristic 
\begin{align}
\label{euler}
\chi_{top}(M)=\chi_{orb}(M)+\frac{|\Gamma|-1}{|\Gamma|},
\end{align}
and a formula for the signature 
\begin{align}
\label{tau}
\tau_{top}(M)=\tau_{orb}(M)-\eta(S^3/\Gamma),
\end{align}
where $\Gamma \subset {\rm{SO}}(4)$ 
is the orbifold group around the fixed point
and $\eta(S^3/\Gamma)$ is the eta-invariant, which in our case is given by
\begin{align}
\eta(S^3/\Gamma)=\frac{1}{|\Gamma|}\sum_{\gamma \neq Id}\Big[\cot(\frac{\theta_1}{2}j)\cot(\frac{\theta_2}{2}j)\Big].
\end{align}
See \cite{Hitchin} for a useful discussion of the 
formulas $\eqref{euler}$ and $\eqref{tau}$.

 Combining formulas $\eqref{euler}$ and $\eqref{tau}$ with the formula 
for the $\Gamma$-Index given in $\eqref{I-G}$, we have
\begin{align}
\label{Gamma-I}
Ind_{\Gamma}=\frac{1}{2}(15\chi_{top}+29\tau_{top})-\frac{15}{2}\bigg(\frac{|\Gamma|-1}{|\Gamma|}\bigg)+\frac{29}{2}\eta(S^3/\Gamma)+\frac{1}{|\Gamma|}\sum_{\gamma\neq Id}\frac{ch_{\gamma}(i^*\sigma)}{ch_{\gamma}(\lambda_{-1}N_{\mathbb{C}})},
\end{align}
where the last term is given by formula $\eqref{ch_g}$.

\section{$\Gamma$-Index for cyclic group actions}
\label{gcyclic}
We consider an orbifold with an isolated singularity having the group action 
$\Gamma_{(q,p)}$ generated by
\begin{align}
\gamma = 
\left( {\begin{array}{*{20}c}
   \cos(\frac{2\pi}{p}) & -\sin(\frac{2\pi}{p}) & 0 & 0 \\
  \sin(\frac{2\pi}{p}) & \cos(\frac{2\pi}{p}) & 0 & 0  \\
   0 & 0 & \cos(\frac{2\pi}{p}q)& -\sin(\frac{2\pi}{p}q) \\
    0 & 0 & \sin(\frac{2\pi}{p}q)& \cos(\frac{2\pi}{p}q)
      \end{array} } \right),
\end{align}
where $p$ and $q$ are relatively prime.  
The cases when $q=1$ and $q=p-1$ have already been resolved in \cite{ViaclovskyIndex}, 
and although we are specifically interested when $1<q<p-1$, we will make use of the sum
\begin{align}
\sum_{\gamma \neq Id} \frac{ch_{\gamma}(i^*\sigma)}{ch_{\gamma}(\lambda_{-1}N_{\mathbb{C}})}
\end{align}
in all cases, and make our computations accordingly.  We begin this section by simplifying our formula for this sum in general:
\begin{align}
\begin{split}
\label{sich}
\sum_{{\gamma} \neq Id} \frac{ch_{\gamma}(i^*\sigma)}{ch_{\gamma}(\lambda_{-1}N_{\mathbb{C}})}&=\sum_{j=1}^{p-1}\Big[-\frac{1}{2}+2(1+\cos(\frac{2\pi}{p}))(1+\cos(\frac{2\pi}{p}q))-\frac{1}{2}\cot(\frac{\pi}{p})\cot(\frac{\pi}{p}q)\Big]\\
&\phantom{==}-\sum_{j=1}^{p-1}\Big[2\cot(\frac{\pi}{p}j)\cot(\frac{\pi}{p}qj)\cos(\frac{2\pi}{p}j)\cos(\frac{2\pi}{p}qj)\Big]\\
&=\sum_{j=1}^{p-1}\Big[\frac{3}{2}+2\cos(\frac{2\pi}{p}j)+2\cos(\frac{2\pi}{p}qj)+ \cos(\frac{2\pi}{p}(q+1)j)\Big]\\
&\phantom{==}+\sum_{j=1}^{p-1}\Big[\cos(\frac{2\pi}{p}(q-1)j)-\frac{1}{2}\cot(\frac{\pi}{p}j)\cot(\frac{\pi}{p}qj) \Big]\\
&\phantom{==}+\sum_{j=1}^{p-1}\Big[-2\cot(\frac{\pi}{p}j)\cot(\frac{\pi}{p}qj)\cos(\frac{2\pi}{p}j)\cos(\frac{2\pi}{p}qj)\Big].
\end{split}
\end{align}
Now, to further simplify our formula for the $\Gamma$-Index, it is necessary to separate into the following cases:

\subsection{$\Gamma$-Index when $1<q<p-1$} Using $\eqref{sich}$, we see that in this case
\begin{align}
\begin{split}
\sum_{{\gamma}\neq Id} \frac{ch_{\gamma}(i^*\sigma)}{ch_{\gamma}(\lambda_{-1}N_{\mathbb{C}})}&=\Big[\frac{3}{2}p-\frac{15}{2}\Big]
-\frac{1}{2}\sum_{j=1}^{p-1}\Big[\cot(\frac{\pi}{p}j)\cot(\frac{\pi}{p}qj)\Big]\\
&\phantom{==}-2\sum_{j=1}^{p-1}\Big[\cot(\frac{\pi}{p}j)\cot(\frac{\pi}{p}qj)\cos(\frac{2\pi}{p}j)\cos(\frac{2\pi}{p}qj)\Big].
\end{split}
\end{align}
Therefore, by combining this with formula $\eqref{Gamma-I}$ for the $\Gamma$-Index, we have
\begin{align}
\begin{split}
\label{N(q,p)}
Ind_{\Gamma}(M)&=\frac{1}{2}(15\chi_{top}+29\tau_{top})-6+\frac{14}{p}\sum_{j=1}^{p-1}\Big[\cot(\frac{\pi}{p}j)\cot(\frac{\pi}{p}qj)\Big]\\
&\phantom{==}-\frac{2}{p}\sum_{j=1}^{p-1}\Big[\cot(\frac{\pi}{p}j)\cot(\frac{\pi}{p}qj)\cos(\frac{2\pi}{p}j)\cos(\frac{2\pi}{p}qj)\Big].
\end{split}
\end{align}

\subsection{$\Gamma$-Index when $q=1$ and $p=2$}
Using $\eqref{sich}$, we see that in this case
\begin{align}
\begin{split}
\sum_{{\gamma}\neq Id} \frac{ch_{\gamma}(i^*\sigma)}{ch_{\gamma}(\lambda_{-1}N_{\mathbb{C}})}=-\frac{1}{2}.
\end{split}
\end{align}
Therefore, by combining this with formula $\eqref{Gamma-I}$ for the $\Gamma$-Index, we have
\begin{align}
\begin{split}
Ind_{\Gamma}(M)&=\frac{1}{2}(15\chi_{top}+29\tau_{top})-4
\end{split}
\end{align}

\subsection{$\Gamma$-Index when $q=1$ and $p>2$}
Using $\eqref{sich}$, we see that in this case
\begin{align}
\begin{split}
\sum_{{\gamma}\neq Id} \frac{ch_{\gamma}(i^*\sigma)}{ch_{\gamma}(\lambda_{-1}N_{\mathbb{C}})}&=\frac{5}{2}p-\frac{15}{2}-\sum_{j=1}^{p-1}\Big[\frac{1}{2}\cot^2(\frac{\pi}{p}j)+2\cot^2(\frac{\pi}{p}j)\cos^2(\frac{2\pi}{p}j)\Big].
\end{split}
\end{align}
Therefore, by combining this with formula $\eqref{Gamma-I}$ for the $\Gamma$-Index,
and the following well-known formula for the Dedekind sum (see \cite{Rademacher2}):
\begin{align}
\frac{1}{4p}\sum_{j=1}^{p-1}\cot^2(\frac{\pi}{p}j) = \frac{1}{12p}(p-1)(p-2),
\end{align}
we have
\begin{align}
\begin{split}
\label{q=1}
Ind_{\Gamma}(M)&=\frac{1}{2}(15\chi_{top}+29\tau_{top})-5+\frac{14}{p}\sum_{j=1}^{p-1}\cot^2(\frac{\pi}{p}j)-\frac{2}{p}\sum_{j=1}^{p-1}\cot^2(\frac{\pi}{p}j)\cos^2(\frac{2\pi}{p}j)\\
&=\frac{1}{2}(15\chi_{top}+29\tau_{top})-5+\frac{12}{p}\sum_{j=1}^{p-1}\cot^2(\frac{\pi}{p}j)+\frac{8}{p}\sum_{j=1}^{p-1}\cos^4(\frac{\pi}{p}j)\\
&=\frac{1}{2}(15\chi_{top}+29\tau_{top})-2-\frac{8}{p}+\frac{12}{p}\sum_{j=1}^{p-1}\cot^2(\frac{\pi}{p}j)\\
&=\frac{1}{2}(15\chi_{top}+29\tau_{top})-2-\frac{8}{p}+\frac{4}{p}(p^2-3p+2)\\
&=\frac{1}{2}(15\chi_{top}+29\tau_{top})+4p-14.
\end{split}
\end{align}

\subsection{$\Gamma$-Index when $q=p-1$ and $p>2$}
Using $\eqref{sich}$, we see that in this case
\begin{align}
\begin{split}
\sum_{{\gamma}\neq Id} \frac{ch_{\gamma}(i^*\sigma)}{ch_g(\lambda_{-1}N_{\mathbb{C}})}&=\frac{5}{2}p-\frac{15}{2}+\sum_{j=1}^{p-1}\Big[\frac{1}{2}\cot^2(\frac{\pi}{p}j)+2\cot^2(\frac{\pi}{p}j)\cos^2(\frac{2\pi}{p}j)\Big].
\end{split}
\end{align}
Therefore, by combining this with formula $\eqref{Gamma-I}$ for the $\Gamma$-Index, we have
\begin{align}
\begin{split}
\label{q=p-1}
Ind_{\Gamma}(M)&=\frac{1}{2}(15\chi_{top}+29\tau_{top})-5-\frac{14}{p}\sum_{j=1}^{p-1}\cot^2(\frac{\pi}{p}j)+\frac{2}{p}\sum_{j=1}^{p-1}\cot^2(\frac{\pi}{p}j)\cos^2(\frac{2\pi}{p}j)\\
&=\frac{1}{2}(15\chi_{top}+29\tau_{top})-8+\frac{8}{p}-\frac{12}{p}\sum_{j=1}^{p-1}\cot^2(\frac{\pi}{p}j)\\
&=\frac{1}{2}(15\chi_{top}+29\tau_{top})-8+\frac{8}{p}-\frac{4}{p}(p^2-3p+2)\\
&=\frac{1}{2}(15\chi_{top}+29\tau_{top})-4p+4.
\end{split}
\end{align}

\section{Non-topological terms in the $\Gamma$-Index}
\label{nontop}
We denote the terms in the $\Gamma$-Index not involving the topological Euler characteristic or topological signature by $N(q,p)$.  Also we change our  notation of the $\Gamma$-Index from $Ind_{\Gamma}$  to $Ind_{(q,p)}$ to reflect the particular group action.  With this new notation we can write the index as
\begin{align}
\label{indn}
Ind_{(q,p)}&=\frac{1}{2}(15\chi_{top}+29\tau_{top})+N(q,p).
\end{align}
In this section we will simplify our formulas for $N(q,p)$.  Also, for the remainder of the paper we will use the following notation.  For two relatively prime positive integers $\alpha < \beta$, denote $\alpha$'s inverse modulo $\beta$ by $\alpha^{-1;\beta}$, and $\beta$'s inverse modulo $\alpha$ by $\beta^{-1;\alpha}$, i.e.
\begin{align}
\alpha\alpha^{-1;\beta}\equiv \text{1 mod $\beta$} \text{ and }
\beta\beta^{-1;\alpha}\equiv \text{1 mod $\alpha$}.
\end{align}

In the cases that $N(q,p)$ is easy to compute we see that
\begin{align}
\label{etc}
N(q,p)=
\begin{cases}
4p-14 &\text{ when $1=q<p-1$}\\
-4p+4 &\text{ when   } q=p-1.
\end{cases}
\end{align}
Note that the case when $q=\pm 1$ and $p=2$ can be actually included in the $q=p-1$ case.  It will be convenient later in paper if we also have these formulas written in terms of sawtooth functions, a cotangent sum and a constant where the sawtooth function is defined to be
\begin{align}
\label{st}
((x))=
\begin{cases}
x-\lfloor x \rfloor -\frac{1}{2} &\text{   when $x\notin \mathbb{Z}$}\\
0 &\text{   when $x\in \mathbb{Z}$}.
\end{cases}
\end{align}
We will include the formulas from $\eqref{etc}$, written in this way, below in Theorem $\ref{N-nonexceptional}$.

To compute $N(q,p)$ in all other cases we will employ the following proposition:
\begin{proposition}
\label{fsp}
\begin{align}
-\frac{1}{2p}\sum_{j=1}^{p-1}\sin(\frac{2\pi}{p}qj)\cot(\frac{\pi}{p}j)=\bigg(\bigg(\frac{q}{p}\bigg)\bigg),
\end{align}
which is the sawtooth function defined in $\eqref{st}$.
\end{proposition}
\begin{proof}
This is due to Eisenstein; see \cite{Apostol}. 
\end{proof}
Now, we have
\begin{theorem}
\label{N-nonexceptional}
When $q\not\equiv (p-1) \text { mod } p$ we have the formula
\begin{align} 
N(q,p)=-6+\frac{12}{p}\sum_{j=1}^{p-1}\cot(\frac{\pi}{p}j)\cot(\frac{\pi}{p}qj)-4\bigg(\bigg(\frac{q^{-1;p}}{p}\bigg)\bigg) -4\bigg(\bigg(\frac{q}{p}\bigg)\bigg),
\end{align}
and when $q\equiv (p-1) \text{ mod } p$ we have the formula
\begin{align}
N(q,p)=N(p-1,p)=-4 -\frac{12}{p}\sum_{j=1}^{p-1}cot^2(\frac{\pi}{p}j)+4\bigg(\bigg(\frac{1}{p}\bigg)\bigg)+4\bigg(\bigg(\frac{1}{p}\bigg)\bigg).
\end{align}
\end{theorem}
\begin{proof}
For the $q=p-1$ case, by examining the formulas in $\eqref{q=1}$ and $\eqref{q=p-1}$, one can easily see that we can also write $N(p-1,p)=-4p+4$ in this way.  Now, consider the $1\leq q<p$ case.  From \eqref{N(q,p)}, we begin by computing
\begin{align}
\begin{split}
N(q,p)&=-6+\frac{14}{p}\sum_{j=1}^{p-1}\Big[\cot(\frac{\pi}{p}j)\cot(\frac{\pi}{p}qj)\Big]\\
&\phantom{==}-\frac{2}{p}\sum_{j=1}^{p-1}\Big[\cot(\frac{\pi}{p}j)\cot(\frac{\pi}{p}qj)\cos(\frac{2\pi}{p}j)\cos(\frac{2\pi}{p}qj)\Big]\\
&=-6+\frac{2}{p}\sum_{j=1}^{p-1}\cot(\frac{\pi}{p}j)\cot(\frac{\pi}{p}qj)\Big[7-\cos(\frac{2\pi}{p}j)\cos(\frac{2\pi}{p}qj)\Big],
\end{split}
\end{align}
and using the identity $\cos (2x) = 1 - 2 \sin^2 (x)$ this expands to
\begin{align}
\begin{split}
&=-6+\frac{2}{p}\sum_{j=1}^{p-1}\cot(\frac{\pi}{p}j)\cot(\frac{\pi}{p}qj)\Big[7-(1-2\sin^2(\frac{\pi}{p}j))(1-2\sin^2(\frac{\pi}{p}qj))\Big]\\
&=-6+\frac{2}{p}\sum_{j=1}^{p-1}\cot(\frac{\pi}{p}j)\cot(\frac{\pi}{p}qj)\Big[6+2\sin^2(\frac{\pi}{p}j)+2\sin^2(\frac{\pi}{p}qj)\Big]\\
&\phantom{==}+\frac{2}{p}\sum_{j=1}^{p-1}\cot(\frac{\pi}{p}j)\cot(\frac{\pi}{p}qj)\Big[-4\sin^2(\frac{\pi}{p}j)\sin^2(\frac{\pi}{p}qj))\Big],
\end{split}
\end{align}
which simplifies to
\begin{align}
\begin{split}
\label{rhsf}
N(q,p)&=
-6+\frac{1}{p}\sum_{j=1}^{p-1}\Big[12\cot(\frac{\pi}{p}j)\cot(\frac{\pi}{p}qj)\Big]+\frac{1}{p}\sum_{j=1}^{p-1}\Big[2\sin(\frac{2\pi}{p}j)\cot(\frac{\pi}{p}qj)\Big]\\
&\phantom{==}+\frac{1}{p}\sum_{j=1}^{p-1}\Big[4\sin(\frac{\pi}{p}qj)\cos(\frac{\pi}{p}qj)\cot(\frac{\pi}{p}j)\Big]\\
&\phantom{==}-\frac{1}{p}\sum_{j=1}^{p-1}\Big[8\sin(\frac{\pi}{p}j)\cos(\frac{\pi}{p}j)\sin(\frac{\pi}{p}qj)\cos(\frac{\pi}{p}qj)\Big].
\end{split}
\end{align}
The fifth term on the right hand side of \eqref{rhsf} sums to zero because
\begin{align}
\begin{split}
\frac{-8}{p}\sum_{j=1}^{p-1}\sin(\frac{\pi}{p}j)\cos(\frac{\pi}{p}j)&\sin(\frac{\pi}{p}qj)\cos(\frac{\pi}{p}qj)=\frac{-4}{p}\sum_{j=1}^{p-1}\sin(\frac{2\pi}{p}j)\sin(\frac{2\pi}{p}qj)\\
&=\frac{-2}{p}\sum_{j=1}^{p-1}\Big[cos(\frac{2\pi}{p}(1-q)j)-cos(\frac{2\pi}{p}(1+q)j)\Big]
=0.
\end{split}
\end{align}
Using Proposition \ref{fsp},
the fourth term on the right hand side of \eqref{rhsf} is
\begin{align*}
\frac{4}{p} \sum_{j=1}^{p-1} \sin(\frac{\pi}{p}qj)\cos(\frac{\pi}{p}qj)\cot(\frac{\pi}{p}j)
&=\frac{2}{p} \sum_{j=1}^{p-1} \sin(2 \frac{\pi}{p}qj)\cot(\frac{\pi}{p}j) 
= -4 \bigg(\bigg( \frac{q}{p}\bigg)\bigg),
\end{align*}
and the third term on the right hand side of \eqref{rhsf} is
\begin{align*}
\frac{2}{p}\sum_{j=1}^{p-1}\sin(\frac{2\pi}{p}j)\cot(\frac{\pi}{p}qj)&=\frac{2}{p}\sum_{j=1}^{p-1}\sin(\frac{2\pi}{p}qq^{-1;p}j)\cot(\frac{\pi}{p}qj)\\
&=\frac{2}{p}\sum_{r=1}^{p-1}\sin(\frac{2\pi}{p}q^{-1;p}r)\cot(\frac{\pi}{p}r)
=-4\bigg(\bigg(\frac{q^{-1;p}}{p}\bigg)\bigg),
\end{align*}
where $r=jq^{-1;p}$, and this finishes the proof. 
\end{proof}
Since the formulas for $N(q,p)$ given in Theorem $\ref{N-nonexceptional}$ are the same in all cases except for when
$q = p-1$, we make the following definition:

\begin{definition}
\label{d-exc} {\em{
A singularity is said to be {\em{exceptional}}
if it results from a $(p-1,p)$-action. Otherwise, it is called
{\em{non-exceptional}}.
}}
\end{definition}

\section{Explicit formula for $N(q,p)$}
\label{explicit}
We begin this section by proving reciprocity formulas for the individual summands of $N(q,p)$.  Then, we use these relations to prove reciprocity formulas for $N(q,p)$, which will later be used to compute $N(q,p)$ explicitly.  Since we have already computed $N(1,p)$, for the simplicity of presentation, we will assume that $q>1$ for the following.  To simplify notation we let $A(q,p) = 48 s(q,p)$, where $s(q,p)$ is the 
Dedekind sum defined in \eqref{Dedekind}.
\begin{proposition}
\label{reci}Writing $p = eq - a$,
the following reciprocity relations are satisfied:
\begin{enumerate}
\item $A(q,p)+A(p,q)=-12+4e-4\frac{a}{q}+4\frac{q}{p}+4\frac{1}{pq}$,
\item $-4\Big(\Big( \frac {q^{-1;p}}{p}\Big)\Big)-4\Big(\Big( \frac {p^{-1;q}}{q}\Big)\Big)=-\frac{4}{pq}$,
\item $-4\Big(\Big( \frac {q}{p}\Big)\Big)-4\Big(\Big( \frac {p}{q}\Big)\Big)=-4\frac{q}{p}+4\frac{a}{q}$.
\end{enumerate}
\end{proposition}
\begin{proof}
By the reciprocity formula for Dedekind sums \cite{Rademacher2}, 
we have that
\begin{align}
\begin{split}
A(q,p)+A(p,q)&=-12+4\Big(\frac{p}{q}+\frac{q}{p}+\frac{1}{pq}\Big)\\
&=-12+4\Big(e-\frac{a}{q}+\frac{q}{p}+\frac{1}{pq}\Big)=-12+4e-4\frac{a}{q}+4\frac{q}{p}+4\frac{1}{pq}.
\end{split}
\end{align}
Next, we have that
\begin{align}
\begin{split}
-4\bigg(\bigg( \frac {q^{-1;p}}{p}\bigg)\bigg)&-4\bigg(\bigg( \frac {p^{-1;q}}{q}\bigg)\bigg)=\Big(-4\frac{q^{-1;p}}{p}+2\Big)+\Big(4\frac{a^{-1;q}}{q}-2\Big)\\
&=-4\frac{q^{-1;p}}{p}+4\frac{a^{-1;q}}{q}=4\frac{-qq^{-1;p}+a^{-1;q}p}{pq}.
\end{split}
\end{align}
Next, using that $q^{-1;p}q=1+ a^{-1;q} p$ (see Proposition \ref{nt}), we have
\begin{align}
\begin{split}
 -4\bigg(\bigg( \frac {q^{-1;p}}{p}\bigg)\bigg)&-4\bigg(\bigg( \frac {p^{-1;q}}{q}\bigg)\bigg)=4\frac{-qq^{-1;p}+a^{-1;q}p}{pq}\\
 &=4\frac{-(1+\alpha p)+a^{-1;q}p}{pq}=-\frac{4}{pq}.
 \end{split}
 \end{align}
Finally, we have
\begin{align}
\begin{split}
-4\bigg(\bigg( \frac {q}{p}\bigg)\bigg)-4\bigg(\bigg( \frac {p}{q}\bigg)\bigg)&=\Big(-4\frac{q}{p}+2\Big)+\Big(4\frac{a}{q}-2\Big)=-4\frac{q}{p}+4\frac{a}{q}.
\end{split}
\end{align}
\end{proof}
Next, we will prove useful reciprocity formulas for $N(q,p)$.  Denote
\begin{align}
\begin{split}
&R^+(q,p)=N(q,p)+N(p,q)\\
&R^-(q,p)=N(-q,p)+N(-p,q).
\end{split}
\end{align}
\begin{proposition} Writing $p=eq-a$ with $0<a<q$,
we have the following formulas:
\label{plus,minus}
\begin{align}
R^+(q,p)=
\begin{cases}
-4 &\text{ when   } q=1 \text{ and } p=2\\
-14 &\text{ when $1<q=p-1$}\\
4p-14 &\text{ when $1=q<p-1$}\\
4e-22 &\text{ when   } p=eq-1\\
4e-24 &\text{  when $2\leq a \leq q-1$},
\end{cases}
\end{align}
and
\begin{align}
R^-(q,p) =
\begin{cases}
-4 &\text{ when   } q=1 \text{ and } p=2\\
-6 &\text{ when $1<q=p-1$}\\
-4p+4 &\text{ when $1=q<p-1$}\\
-4e+2 &\text{  when $p=eq-(q-1)$ and $1<q<p-1$}\\
 -4e &\text{  when $1\leq a \leq q-2$ and $2<q$}.
\end{cases}
\end{align}
\end{proposition}
\begin{proof}
The first three formulas for both $R^+(q,p)$ and $R^-(q,p)$ are easily computable from the cases where $N(q,p)$ is easy to compute.  
Denote by $C_{(\alpha,\beta)}$ the constant term in $N(\alpha,\beta)$, so
\begin{align}
\label{cab}
C_{(\alpha,\beta)}=
\begin{cases}
-6 &\text{   for a non-exceptional singularity}\\
-4 &\text{   for an exceptional singularity}.
\end{cases}
\end{align}
For the case when $p=eq-a$, where $1\leq a <q-1$, we have that
\begin{align*}
R^+(q,p)=N(q,p)+N(p,q)&=\bigg[C_{(q,p)}+A(q,p)-4\bigg(\bigg(\frac{q^{-1;p}}{p}\bigg)\bigg) -4\bigg(\bigg(\frac{q}{p}\bigg)\bigg)\bigg]\\
&\phantom{==}+\bigg[C_{(p,q)}+A(p,q)-4\bigg(\bigg(\frac{p^{-1;q}}{q}\bigg)\bigg)-4\bigg(\bigg(\frac{p}{q}\bigg)\bigg)\bigg].
\end{align*}
Then, by Proposition $\ref{reci}$, we see that
\begin{align*}
R^+(q,p)&=C_{(q,p)}+C_{(p,q)}+\Big[-12+4e-4\frac{a}{q}+4\frac{q}{p}+\frac{4}{pq}\Big]+\Big[-\frac{4}{pq}-4\frac{q}{p}+4\frac{a}{q}\Big]\\
&=4e+C_{(q,p)}+C_{(p,q)}-12,
\end{align*}
which proves the reciprocity formulas in each respective case.
The proof for $R^-(q,p)$ is similar and is omitted. 
\end{proof}
We next use the above reciprocity relations 
to recursively compute an explicit formula for $N(q,p)$:
\begin{theorem}For $q$ and  $p$ and relatively prime, we have 
\label{N=}
\begin{align}
\label{nqpform}
N(q,p) =
\begin{cases}
\displaystyle \sum_{i=1}^{k}4e_i-12k-2  &\text{   when $q\not\equiv (p-1)$ mod $p$} \\
\displaystyle \sum_{i=1}^{k}4e_i-12k=-4p+4 &\text{   when $q\equiv (p-1)$ mod $p$},
\end{cases}
\end{align}
where $k$ and $e_i$, $1 \leq i \leq k$, were defined 
above in the modified Euclidean algorithm~\eqref{mea}.
\end{theorem}
\begin{proof}
We have already proved the second case in $\eqref{etc}$, and we will now prove the first case, so we only need consider $q\not\equiv (p-1) \text{ mod } p$.  Since our formulas 
only depend upon $q$ mod $p$, we can assume that $1\leq q<p-1$.  We begin by using 
Proposition \ref{plus,minus} to compute $N(q,p)$ as follows:
\begin{align*}
N(q,p)&=R^+(q,p)-N(p,q)\\
&=R^+(q,p)-N(e_1q-a_1,q)\\
&=R^+(q,p)-N(-a_1,q)\\
&=R^+(q,p)-N(-a_1,q)-N(-q,a_1)+N(-q,a_1)\\
&=R^+(q,p)-R^-(a_1,q)+N(a_2,a_1)+N(a_1,a_2)-N(a_1,a_2)\\
&=R^+(q,p)-R^-(a_1,q)+R^+(a_2,a_1)-N(-a_3,a_2).
\end{align*}
Continuing this iteratively, we arrive at the formula
\begin{align*}
N(q,p)
&=\sum_{i=1}^{r+1}4e_i-24\Big\lceil \frac{r+1}{2} \Big\rceil \\
&\phantom{==}+\Big[(-1)^{r+1}R^{(-1)^{r+1}}(a_{r+1},a_r)+(-1)^{r+2}N((-1)^{r+2}a_{r+2},a_{r+1})\Big],
\end{align*}
where $a_r=e_{r+2}a_{r+1}-1$ or $a_r=e_{r+2}a_{r+1}-(a_{r+1}-1)$.  It is only necessary to consider the four following cases:
\begin{enumerate}
\item When $r+2$  is even and $a_{r+2}=1$:\\ 
\phantom{======}$N(q,p)=\displaystyle        \sum_{i=1}^{r+3}4e_i-12(r+2)-14$.
\item When $r+2$  is odd and $a_{r+2}=1$:\\
\phantom{======} $N(q,p)= \displaystyle       \sum_{i=1}^{r+3}4e_i-12(r+1)-26$.
\item When $r+2$  is even and $a_{r+2}=a_{r+1}-1$:\\
 \phantom{======}$N(q,p)= \displaystyle        \sum_{i=1}^{r+2}4e_i-4a_{r+1}-12(r+2)+2$.
\item When $r+2$  is odd and $a_{r+2}=a_{r+1}-1$:\\
\phantom{======} $N(q,p)= \displaystyle        \sum_{i=1}^{r+2}4e_i-4a_{r+1}-12(r+1)-10$.
\end{enumerate}
The formulas for $N(q,p)$ in each case are a direct consequence of formula $\eqref{etc}$ and Proposition $\ref{plus,minus}$.
In case (1) and case (2), $k=r+3$.  So written in terms of $k$ we have
\begin{align}
N(q,p) =  \sum_{i=1}^{k}4e_i-12k-2,
\end{align}
for both cases.  Now, in case (3), $k=(a_{r+1}-1)+(r+2)$ and $e_i=2$ for $i\geq r+3$.  Therefore we can check that 
\begin{align}
\begin{split}
\sum_{i=1}^{k}4e_i-12k-2&=\bigg[\sum_{i=1}^{r+2}4e_i-12(r+2)\bigg]+\bigg[\sum_{i=r+3}^{k}4e_i-12(a_{r+1}-1)-2\bigg]\\
&=\sum_{i=1}^{r+2}4e_i-12(r+2)-4a_{r+1}+2=N(q,p).
\end{split}
\end{align}
Finally, in case (4), $k=(a_{r+1}-1)+(r+2)$ and $e_i=2$ for $i\geq r+3$,
and the result holds similarly.
\end{proof}
Theorem \ref{mainthm} is then a trivial consequence of Theorem \ref{N=}
and \eqref{indn}.

\begin{remark}{\em
Ashikaga and Ishizaka prove a recursive formula for the 
Dedekind sum in  \cite[Theorem 1.1]{AshiIshi}, which is 
equivalent to Theorem \ref{N=}. However, our proof is more elementary 
and relies only on the reciprocity law for Dedekind sums. 
We will also need to use 
Proposition \ref{plus,minus} below in Section \ref{wpssec}. 
}
\end{remark}

\section{Index on Calderbank-Singer spaces }
\label{CSindex}
In this section, we prove the results regarding the Calderbank-Singer metrics.
Let $k$ and $k'$ be the lengths of the modified Euclidean algorithm for $(q,p)$ and $(p-q,p)$ respectively.
\begin{proof}[Proof of Theorem \ref{CSthm}]
It follows from \eqref{intersect} that
the compactified Calderbank-Singer space $(\hat{X}, \hat{g})$
satisfies
$\tau_{top}(\hat{X})=-k$ and $\chi_{top}(\hat{X})=k+2$,
so for a $(p-q,p)$-action when $q\neq 1$, the index is
\begin{align*}
Ind(\hat{X}, \hat{g})&=\frac{1}{2}(15\chi_{top}+29\tau_{top})+N(q,p)
=[-7k+15]+\Big[\sum_{i=1}^{k'}4e_i'-12k'-2\Big].
\end{align*}
We next use a 4-dimensional $(q,p)$-football, denoted by $S^4_{(q,p)}$, to relate
$k$ and $k'$. This is defined using the $\Gamma_{(p,q)}$ action, 
acting as rotations around $x_5$-axis:
\begin{align}
S^4_{(q,p)}=S^4/\Gamma_{(q,p)}.
\end{align}
This quotient is an orbifold with two singular points,
one of $(q,p)$-type, and the other of $(-q,p)$-type.
Since $\chi_{top}(S^4_{(q,p)})=2$ and $\tau_{top}(S^4_{(q,p)})=0$,
the index of \eqref{thecomplex} on $S^4_{(q,p)}$ with the round metric $g_S$ is
\begin{align}
Ind(S^4_{(q,p)},g_S)=
3 \text{   for $1<q<p-1$}.
\end{align}
Using the formula
\begin{align}
Ind(S^4_{(q,p)},g_S)=\frac{1}{2}(15\chi_{top}+29\tau_{top})+N(q,p)+N(-q,p),
\end{align}
and Theorem \ref{N=}, we have
\begin{align}
\begin{split}
-12&=N(q,p)+N(-q,p)=N(q,p)+N(p-q,p)\\
&=\Big [\sum_{i=1}^{k}4e_i-12k-2\Big]+\Big[\sum_{i=1}^{k'}4e_i'-12k'-2\Big],
\end{split}
\end{align}
which yields the formula
\begin{align}
k'=\frac{1}{12}\Big(8+\sum_{i=1}^{k}4e_i+\sum_{i=1}^{k'}4e_i'-12k\Big).
\end{align}
Then, substituting this for $k$ in $Ind(\hat{X}, \hat{g})$ gives
\begin{align}
\begin{split}
Ind(\hat{X}, \hat{g})&=[-7k+15]+\Big[\sum_{i=1}^{k'}4e_i'-\Big(8+\sum_{i=1}^{k}4e_i+\sum_{i=1}^{k'}4e_i'-12k
\Big)-2\Big]\\
&=5k+5-\sum_{i=1}^{k}4e_i.
\end{split}
\end{align}

Next, when $q=1$, we have $k =1$, so the index is
\begin{align}
Ind(\hat{X}, \hat{g}) =[-7k+15]+[-4p+4] =-4p+12.
\end{align}
\end{proof}

\begin{proof}[Proof of Theorem \ref{CScor}]
Calderbank-Singer showed that their toric metrics come in 
families of dimension $k-1$. It was proved 
by Dominic Wright that the moduli space of toric anti-self-dual 
metrics on the orbifolds are of dimension exactly $k -1$
\cite[Corollary 1.1]{Wright}. So as long as we show the moduli 
space is strictly larger than $k-1$, there must be non-toric
deformations.

The $(1,2)$ case is the Eguchi-Hanson metric which has no 
deformations. 
For $q=1$ and $p > 2$, the $(1,p)$ type Calderbank-Singer spaces are exactly 
the LeBrun negative mass metrics on $\mathcal{O}(-p)$ found in 
\cite{LeBrunnegative}. 
It was shown in \cite{HondaOn} for $p = 3$,
the moduli space of these metrics is of dimension $1$ so the 
result is true since $1 > 0 = k - 1$. For $p \geq 4$, 
by \cite[Theorem 1.9]{ViaclovskyIndex}, the moduli space has 
dimension at least $4p - 12 > 0$ (in fact the dimension is 
exactly $4p-12$, see \cite[Theorem~1.1]{HondaOn}).
So the result holds for $q = 1$ and $p \geq 3$. 
We also mention that \cite[Theorem~1.1]{HondaOn} determines
exactly the identity component of the automorphism groups of 
the deformations. 

 Next, assume that $q = p-1$. In this case, the metrics are
hyperk\"ahler, and correspond to toric multi-Eguchi-Hason metrics. 
In this case, the moduli space of all hyperk\"ahler metrics is known to be 
exactly of dimension $3(k-1)$. 

 Next, we assume that $1 < q < p-1$. 
As mentioned in the Introduction, 
from \cite[Theorem 4.2]{LeBrunMaskit}, we know that $\dim(H^2(\hat{X},\hat{g}))=0$.
Also, $\dim(H^0) = 2,$ since the metrics are toric and $q >1$. Therefore
\begin{align}
\dim(H^1)=-Ind(\hat{X}, \hat{g}) +\dim(H^0) = -Ind(\hat{X}, \hat{g}) +2.
\end{align}
When $q \neq 1$, we have that
\begin{align}
-Ind =-5k-5+\sum_{i=1}^{k}4e_i.
\end{align}
Since $e_i \geq 2 $ for all $i$ and since $q < p-1$, then 
$e_j \geq 3$ for some $j$, $1 \leq j \leq k$.  Therefore
\begin{align}
\dim(H^1)\geq 3k + 1.
\end{align}
The actual moduli space is locally isomorphic to $H^1 / H^0$, so it has dimension 
at least $3k - 1 > 3(k-1)$. 
\end{proof}

\section{Index on weighted projective spaces}
\label{wpssec}
In this section we will study the index of the complex \eqref{thecomplex} 
at the Bochner-K\"ahler metrics of Bryant {\em{with reversed orientation to make 
them anti-self-dual}}.
This reversal of orientation makes the orbifold points have 
orientation-reversing conjugate actions as follows:
\begin{enumerate}
\item Around [1,0,0] there is a $(-q^{-1;r}p,r)$-action.
\item Around [0,1,0] there is a $(-p^{-1;q}r,q)$-action.
\item Around [0,0,1] there is a $(-r^{-1;p}q,p)$-action.
\end{enumerate}

In the next subsection, we will present some elementary number theoretic 
propositions that we will use throughout our computations.
After that,  we will prove 
crucial reciprocity laws for sawtooth functions relating $r, q$ and $p$ 
and then employ these to prove our main formula for the index.  
Finally, we use this formula to prove Theorem~\ref{wpsthm}.

\subsection{Elementary number theoretic preliminaries}

Recall that for two relatively prime positive integers $1<\alpha < \beta$, that we denote $\alpha$'s inverse modulo $\beta$ by $\alpha^{-1;\beta}$, and $\beta$'s inverse modulo $\alpha$ by $\beta^{-1;\alpha}$.  Since $\alpha < \beta$ we can write
\begin{align}
\beta=e\alpha-a,
\end{align}
where $e$ and $a$ are positive integers with $a<\alpha$.  Then we have the following proposition:
\begin{proposition}
\label{nt}
We have the following identities:
\begin{enumerate}
\item $\beta^{-1;\alpha}=\alpha-a^{-1;\alpha}$
\item $\alpha \alpha^{-1;\beta}=1+a^{-1;\alpha}\beta$.
\end{enumerate}
\end{proposition}
\begin{proof}
To prove the first identity, recall that $\beta=e\alpha-a$ so
\begin{align*}
\beta (\alpha-a^{-1;\alpha})&=(e\alpha-a)(\alpha-a^{-1;\alpha})
=e\alpha^2-e\alpha a^{-1;\alpha}-a\alpha+aa^{-1;\alpha}\equiv 1 \mod \alpha.
\end{align*}
This proves the first identity because $\alpha-a^{-1;\alpha}<\alpha$ and the multiplicative inverses are unique.  

To prove second identity we first write
\begin{align}
\alpha\alpha^{-1;\beta}=1+X\beta.
\end{align}
Since $1<\alpha$ we know $X$ must be a positive integer.  We can then solve for
\begin{align}
\beta=\frac{\alpha \alpha^{-1;\beta}-1}{X}.
\end{align}
Therefore $\beta=\frac{\alpha \alpha^{-1;\beta}-1}{X}=e\alpha-a$, so
\begin{align}
\alpha\alpha^{-1;\beta}-1=e\alpha X-aX,
\end{align}
from which we see that
\begin{align}
aX=\alpha(eX-\alpha^{-1;\beta})+1,
\end{align}
so $aX\equiv \text{1 mod $\alpha$}$.  This proves the second identity because $X=\frac{\alpha \alpha^{-1;\beta}-1}{\beta}<\alpha$ and multiplicative inverses are unique.
\end{proof}
For convenience, we define the fractional part of $x$ by 
\begin{align}
\{ x \} = x - \lfloor x \rfloor.
\end{align}
We will use the following proposition extensively in the next section, 
the proof is elementary:
\begin{proposition}
\label{saw}
For any real $\alpha$ and $\beta$, both non-integral,
\begin{align}
((\alpha+\beta))=
\begin{cases}
 ((\alpha))+((\beta))+\frac{1}{2} &\text{   when $\{\alpha\}+\{\beta\}<1$}\\
 ((\alpha))+((\beta))-\frac{1}{2} &\text{   when $\{\alpha\}+\{\beta\}>1$}\\
 0 &\text{   when $\{\alpha\}+\{\beta\}=1$}.
 \end{cases}
\end{align}
\end{proposition}

\subsection{Reciprocity formulas for sawtooth functions}
Let $r<q<p$ and write:
\begin{align}
\begin{split}
&p=e_{pr}r-a_{pr}\\
&p=e_{pq}q-a_{pq}\\
&q=e_{qr}r-a_{qr}.
\end{split}
\end{align}
We have the following identities from Proposition $\ref{nt}$ (1):
\begin{align}
\begin{split}
&p^{-1;r}=r-a^{-1;r}_{pr}\\
&p^{-1;q}=q-a^{-1;q}_{pq}\\
&q^{-1;r}=r-a^{-1;r}_{qr}
\end{split}
\end{align}
and from Proposition $\ref{nt}$ (2):
\begin{align}
\begin{split}
&rr^{-1;p}=1+a^{-1;r}_{pr}p\\
&rr^{-1;q}=1+a^{-1;r}_{qr}q\\
&qq^{-1;p}=1+a^{-1;q}_{pq}p.
\end{split}
\end{align}
We now use these identities to prove reciprocity laws for the sawtooth function.  These reciprocity laws will be broken up into two theorems where the first is independent of $r+q$ in relation to $p$ and the second is dependent.
\begin{theorem}
\label{2}
We have the following reciprocity relations:
\begin{enumerate}
\item $\Big(\Big(\frac{qp^{-1;r}}{r}\Big)\Big)+ \Big(\Big(\frac{qr^{-1;p}}{p}\Big)\Big)=\frac{q}{pr}$\\
\item $\Big(\Big(\frac{rp^{-1;q}}{q}\Big)\Big)+ \Big(\Big(\frac{rq^{-1;p}}{p}\Big)\Big)=\frac{r}{pq}$.
\end{enumerate}
\end{theorem}

\begin{proof}
Consider the first reciprocity relation.  We have
\begin{align}
\begin{split}
\bigg(\bigg(\frac{qp^{-1;r}}{r}\bigg)\bigg)+ \bigg(\bigg(\frac{qr^{-1;p}}{p}\bigg)\bigg)&=\bigg(\bigg(\frac{q(r-a^{-1;r}_{pr})}{r}\bigg)\bigg)+\bigg(\bigg(\frac{rr^{-1;p}q}{pr}\bigg)\bigg)\\
&=\bigg(\bigg(q-\frac{qa^{-1;r}_{pr}}{r}\bigg)\bigg)+\bigg(\bigg(\frac{q}{pr}+\frac{qa^{-1;r}_{pr}}{r}\bigg)\bigg)\\
&=-\bigg(\bigg(\frac{qa^{-1;r}_{pr}}{r}\bigg)\bigg)+\bigg(\bigg(\frac{q}{pr}+\frac{qa^{-1;r}_{pr}}{r}\bigg)\bigg)
\end{split}
\end{align}
Now, we can write $\frac{qa^{-1;r}_{pr}}{r}=X+\frac{C}{r}$ where $0<X$ and $0<C<r$ are positive integers so that
\begin{align}
\label{1}-\bigg(\bigg(\frac{qa^{-1;r}_{pr}}{r}\bigg)\bigg)+\bigg(\bigg(\frac{q}{pr}+\frac{qa^{-1;r}_{pr}}{r}\bigg)\bigg)&=
-\bigg(\bigg(\frac{C}{r}\bigg)\bigg)+\bigg(\bigg(\frac{q}{pr}+\frac{C}{r}\bigg)\bigg).
\end{align}
Since $0<C<r$ we know that 
\begin{align}
\frac{q}{pr}+\frac{C}{r}\leq \frac{q}{pr}+\frac{r-1}{r}=\frac{q}{pr}+\frac{pr-p}{pr}<1,
\end{align}
because $q<p$, which implies that
\begin{align}
\bigg\{\frac{p}{qr}\bigg\}+\bigg\{\frac{C}{r}\bigg\}<1.
\end{align}
Therefore, by Proposition $\ref{saw}$, we can separate the second sawtooth function to get
\begin{align}
\bigg(\bigg(\frac{q}{pr}+\frac{C}{r}\bigg)\bigg)=\bigg(\bigg(\frac{q}{pr}\bigg)\bigg)+\bigg(\bigg(\frac{C}{r}\bigg)\bigg)+\frac{1}{2}.
\end{align}
Putting this back into $\eqref{1}$ we see that
\begin{align}
\begin{split}
\bigg(\bigg(\frac{qp^{-1;r}}{r}\bigg)\bigg)&+ \bigg(\bigg(\frac{qr^{-1;p}}{p}\bigg)\bigg)=-\bigg(\bigg(\frac{C}{r}\bigg)\bigg)+\bigg(\bigg(\frac{q}{pr}+\frac{C}{r}\bigg)\bigg)\\
&=-\bigg(\bigg(\frac{C}{r}\bigg)\bigg)+\bigg(\bigg(\frac{q}{pr}\bigg)\bigg)+\bigg(\bigg(\frac{C}{r}\bigg)\bigg)+\frac{1}{2}=\bigg(\bigg(\frac{q}{pr}\bigg)\bigg)+\frac{1}{2}\\
&=\frac{q}{pr}-\Big\lfloor \frac{q}{pr}\Big \rfloor=\frac{q}{pr}.
\end{split}
\end{align}
The proof of the second reciprocity relation exactly follows the proof of the first.
\end{proof}

The next theorem gives a similar reciprocity relation, but it is dependent upon $r+q$ in relation to $p$.
\begin{theorem}
\label{3}
For $1<r<q<p$ we have the following reciprocity relation:
\begin{align}
\bigg(\bigg(\frac{q^{-1;r}p}{r}\bigg)\bigg)+\bigg(\bigg(\frac{r^{-1;q}p}{q}\bigg)\bigg)=
\begin{cases}
\frac{p}{qr} &\text{when $r+q>p$}\\
\frac{p}{qr}-1 &\text{when $r+q=p$}\\
\frac{p}{qr}-\lfloor \frac{p}{qr}\rfloor &\text{when $r+q<p$}\\
&\text{and $\Big\{\frac{p}{qr}\Big\}<\Big\{\frac{q^{-1;r}p}{r}\Big\}$}\\
\frac{p}{qr}-\lfloor \frac{p}{qr}\rfloor-1 &\text{when $r+q<p$}\\
&\text{and $\Big\{\frac{p}{qr}\Big\}>\Big\{\frac{q^{-1;r}p}{r}\Big\}$}.
\end{cases}
\end{align}
\end{theorem}

\begin{proof}
We begin in a similar way to the proof of Theorem \ref{2}:
\begin{align}
\begin{split}
\bigg(\bigg(\frac{q^{-1;r}p}{r}\bigg)\bigg)+\bigg(\bigg(\frac{r^{-1;q}p}{q}\bigg)\bigg)&=\bigg(\bigg(\frac{p(r-a^{-1;r}_{qr})}{r}\bigg)\bigg)+\bigg(\bigg(\frac{(1+a^{-1;r}_{qr}q)p}{qr}\bigg)\bigg)\\
&=-\bigg(\bigg(\frac{a^{-1;r}_{qr}p}{r}\bigg)\bigg)+\bigg(\bigg(\frac{p}{qr}+\frac{a^{-1;r}_{qr}p}{r}\bigg)\bigg).
\end{split}
\end{align}
Now, we can write $\frac{a^{-1;r}_{qr}p}{r}=X+\frac{C}{r}$ where $0<X$ and $0<C<r$ are positive integers so that
\begin{align}
-\bigg(\bigg(\frac{a^{-1;r}_{qr}p}{r}\bigg)\bigg)+\bigg(\bigg(\frac{p}{qr}+\frac{a^{-1;r}_{qr}p}{r}\bigg)\bigg)&=-\bigg(\bigg(\frac{C}{r}\bigg)\bigg)+\bigg(\bigg(\frac{p}{qr}+\frac{C}{r}\bigg)\bigg).
\end{align}
The same argument that we used in the previous proof to split up the second sawtooth function will no longer work because $p>q$, which could allow $\frac{p}{qr}+\frac{C}{r}>1$.  Lets consider the first case of this reciprocity relation when $r+q>p$.  In this case we know that $p<rq$ so $p<2q$ since $1<r<q$.  Now, we will show that $C\leq r-2$ and use this to prove the first case.  Write $p$ as
\begin{align}
p=k_{pr}r+m_{pr},
\end{align}
where $k_{pr}=e_{pr}-1$ and $m_{pr}=r-a_{pr}$ are positive integers.  We know that $\frac{C}{r}$ is going to be the fractional part of $\frac{pa^{-1;r}_{qr}}{r}$ which equals the fractional part of $\frac{m_{pr}a^{-1;r}_{qr}}{r}$.  If this equals $\frac{r-1}{r}$ then $m_{pr}a^{-1;r}_{pr}\equiv -1 (r)$ and therefore $m_{pr}=r-a_{qr}$, because multiplicative inverses are unique, which implies that $a_{pr}=a_{qr}$ (denote this value by $A$).  So we have that
\begin{align}
e_{pr}r-A=p<q+r=(e_{qr}r-A)+r=(e_{qr}+q)r-A,
\end{align}
which is a contradiction because $e_{pr}\geq e_{qr}+1$.  Therefore $\frac{C}{r}\leq \frac{r-2}{r}$, so
\begin{align}
\frac{p}{qr}+\frac{C}{r}\leq \frac{p}{qr}+\frac {r-2}{r}=\frac{p}{qr}+\frac{qr-2q}{qr}<1,
\end{align}
which implies that
\begin{align}
\bigg\{\frac{p}{qr}\bigg\}+\bigg\{\frac{C}{r}\bigg\}<1.
\end{align}
Therefore, by Proposition $\ref{saw}$, we can separate the second sawtooth function to get
\begin{align}
\begin{split}
-\bigg(\bigg(\frac{C}{r}\bigg)\bigg)+\bigg(\bigg(\frac{p}{qr}+\frac{C}{r}\bigg)\bigg)&=-\bigg(\bigg(\frac{C}{r}\bigg)\bigg)+\bigg(\bigg(\frac{p}{qr}\bigg)\bigg)+\bigg(\bigg(\frac{C}{r}\bigg)\bigg)+\frac{1}{2}\\
&=\bigg(\bigg(\frac{p}{qr}\bigg)\bigg)+\frac{1}{2}=\frac{p}{qr}-\Big\lfloor \frac{p}{qr}\Big \rfloor=\frac{p}{qr},
\end{split}
\end{align}
which proves the first case.  

Now, consider the second case when $r+q=p$.  In this case we see that
\begin{align}
\begin{split}
\bigg(\bigg(\frac{p}{qr}+\frac{pa^{-1;r}_{qr}}{r}\bigg)\bigg)&=\bigg(\bigg(\frac{p}{qr}+\frac{(r+q)a^{-1;r}_{qr}}{r}\bigg)\bigg)\\
&=\bigg(\bigg(\frac{p}{qr}-\frac{a_{qr}a^{-1;r}_{qr}}{r}\bigg)\bigg)=\bigg(\bigg(\frac{p}{qr}-\frac{1}{r}\bigg)\bigg)\\
&=\Big(\frac{p}{qr}-\frac{1}{r}\Big)-\Big\lfloor \frac{p}{qr}-\frac{1}{r} \Big\rfloor -\frac{1}{2}
=\frac{1}{q}-\frac{1}{2}.
\end{split}
\end{align}
Then, we compare this to
\begin{align}
\begin{split}
\bigg(\bigg(&\frac{p}{qr}\bigg)\bigg)+\bigg(\bigg(\frac{pa^{-1;r}_{qr}}{r}\bigg)\bigg)+\frac{1}{2}=\bigg(\bigg(\frac{p}{qr}\bigg)\bigg)+\bigg(\bigg(\frac{(r+q)a^{-1;r}_{qr}}{r}\bigg)\bigg)+\frac{1}{2}\\
&=\bigg(\bigg(\frac{p}{qr}\bigg)\bigg)+\bigg(\bigg(\frac{-a_{qr}a^{-1;r}_{qr}}{r}\bigg)\bigg)+\frac{1}{2}=\bigg(\bigg(\frac{p}{qr}\bigg)\bigg)-\bigg(\bigg(\frac{1}{r}\bigg)\bigg)+\frac{1}{2}\\
&=\Big(\frac{p}{qr}-\Big\lfloor \frac{p}{qr}\Big \rfloor-\frac{1}{2}\Big)-\Big(\frac{1}{r}-\Big\lfloor \frac{1}{r}\Big \rfloor -\frac{1}{2}\Big)+\frac{1}{2}\\
&=\frac{p}{qr}-\frac{1}{r}+\frac{1}{2}=\frac{1}{q}+\frac{1}{2}
=\bigg(\bigg(\frac{p}{qr}+\frac{pa^{-1;r}_{qr}}{r}\bigg)\bigg)+1.
\end{split}
\end{align}
Therefore we see that
\begin{align}
\begin{split}
-\bigg(\bigg(\frac{a^{-1;r}_{qr}p}{r}\bigg)\bigg)&+\bigg(\bigg(\frac{p}{qr}+\frac{a^{-1;r}_{qr}p}{r}\bigg)\bigg)\\
&=-\bigg(\bigg(\frac{a^{-1;r}_{qr}p}{r}\bigg)\bigg)+\bigg(\bigg(\frac{p}{qr}\bigg)\bigg)+\bigg(\bigg(\frac{pa^{-1;r}_{qr}}{r}\bigg)\bigg)-\frac{1}{2}\\
&=\bigg(\bigg(\frac{p}{qr}\bigg)\bigg)-\frac{1}{2}=\frac{p}{qr}-\Big\lfloor \frac{p}{qr} \Big\rfloor -1=\frac{p}{qr}-1,
\end{split}
\end{align}
which proves the second case.

To prove the third and fourth cases we begin once again by using that
\begin{align}
\bigg(\bigg(\frac{q^{-1;r}p}{r}\bigg)\bigg)+\bigg(\bigg(\frac{r^{-1;q}p}{q}\bigg)\bigg)=\bigg(\bigg(\frac{q^{-1;r}p}{r}\bigg)\bigg)+\bigg(\bigg(\frac{p}{qr}-\frac{q^{-1;r}p}{r}\bigg)\bigg).
\end{align}
Notice that 
\begin{align}
\Big\{\frac{p}{qr}\Big\}+\Big\{\frac{-q^{-1;r}p}{r}\Big\}=\Big\{\frac{p}{qr}\Big\}+1-\Big\{\frac{q^{-1;r}p}{r}\Big\},
\end{align}
which is never equal to one because $r$, $q$ and $p$ are relatively prime.  Now, the rest of the proof of the third and fourth cases follows directly from Proposition $\ref{saw}$.
\end{proof}

\subsection{$\Gamma$-Index for weighted projective spaces}
First, recall Definition $\ref{d-exc}$: singularities resulting from a $(p-1,p)$-action are said to be {\em{exceptional}} and otherwise they are called {\em{non-exceptional}}.
Consider the case when $1<r<q<p$ so that there are three 
singularities.
Before giving theorems concerning the index, we will first examine what type singularities, non-exceptional or exceptional, are admitted around each orbifold point in the cases when $r+q>p$, $r+q=p$ and $r+q<p$.

\begin{proposition}
\label{gr}
When $r+q>p$ all three singularites are non-exceptional.
When $r+q=p$ we have that
\begin{enumerate}
\item The singularity at $[1,0,0]$ is always exceptional.
\item The singularity at $[0,1,0]$ is always exceptional.
\item The singularity at $[0,0,1]$ is non-exceptional and comes from a $(1,p)$-action.
\end{enumerate}
When $r+q<p$ we have that
\begin{enumerate}
\item  The singularity at $[1,0,0]$ is exceptional if and only if $p\equiv q\text{ mod $r$}$.
\item  The singularity at $[0,1,0]$ is exceptional if and only if $p\equiv r\text{ mod $q$}$.
\item  The singularity at $[0,0,1]$ is always non-exceptional.
\end{enumerate}
\end{proposition}
\begin{proof}
At $[1,0,0]$ the $(-q^{-1;r}p,r)$-action is equivalent to a $(-a^{-1;r}_{qr}a_{pr},r)$-action, and this is equivalent to a $(r-1,r)$-action if and only if $a_{pr}=a_{qr}$.  
If $ r + q > p$, suppose that $a_{pr}=a_{qr}$, then 
\begin{align}
 p=e_{pr}r-a_{qr}, \text{ and } q=e_{qr}r-a_{qr},
\end{align}
so $p<q+r=(e_{qr}+1)r-a_{qr}$, which is a contradiction because $e_{pr}\geq e_{qr}+1$.  
If $r+q=p$ we have that
\begin{align}
p=q+r=(e_{qr}+1)r-a_{qr},
\end{align}
so we see that $a_{pr}=a_{qr}$ since $a_{qr}<r$.  
If $ r + q < p$, then this happens if and only if $p\equiv q\text{ mod $r$}$.

At $[0,1,0]$, by Remark \ref{actrem}, the $(-p^{-1;q}r,q)$-action is equivalent to
a $(-r^{-1;q}p,q)$-action. This is equivalent to 
a $(r^{-1;q}a_{pq},q)$-action, which is equivalent to a $(q-1,q)$-action if and only if $a_{pq}r^{-1;q}\equiv -1 \mod q$, which would imply that $a_{pq}=q-r$. 
If $ r + q > p$, suppose that $a_{pq}=q-r$, then
\begin{align}
p=2q-a_{pq}=2q-(q-r) =q+r,
\end{align}
which is a contradiction because $r+q>p$. If $r+q=p$, we have that
\begin{align}
p=2q-(q-r),
\end{align}
so we see that $a_{pq}=q-r$.  If $ r + q < p$ then this happens 
if and only if  $p\equiv r \mod q$.

At $[0,0,1]$ the $(-r^{-1;p}q,p)$-action is equivalent to a $(p-1,p)$-action if and only if $r^{-1;p}q\equiv 1 \mod p$. If $ r + q > p$, this condition   
would imply that $q=r$, which is a contradiction. 
If $r+q=p$ then the  $(-r^{-1;p}q,p)$-action is obviously equivalent to a $(1,p)$-action 
since $q = p - r$. 
If $ r + q < p$ then $r^{-1;p}q\equiv 1 \mod p$ occurs if and only if $q=r$, 
but $q>r$ so this can never happen.
\end{proof}
In the case $r + q < p$, we can add the following:
\begin{proposition}
\label{lee}
When $r+q<p$ and the singularities at $[1,0,0]$ and $[0,1,0]$ are both exceptional, we have that $p=Xqr+r+q$ for some integer $X$, and
\begin{align}
\bigg\{\frac{p}{qr} \bigg\}>\bigg\{\frac{q^{-1;r}p}{r}\bigg\}.
\end{align}
\end{proposition}
\begin{proof}
Since the singularities around $[1,0,0]$ and $[0,1,0]$ are both exceptional, from Proposition $\ref{gr}$ we know that
\begin{align}
p\equiv q\text{ mod } r, \text{ and } p\equiv r \text{ mod } q.
\end{align}
Therefore, we can write
\begin{align}
p=Y_1q+r=Y_2r+q,
\end{align}
and solve for
\begin{align}
r=\frac{Y_1-1}{Y_2-1}q,
\end{align}
which implies that $qX=Y_2-1$ for some $X$ in $\mathbb{Z}$, since $q$ and $r$ are relatively prime.  Then solving for $Y_2=qX+1$ we see that
\begin{align}
p=(qX+1)r+q =Xqr+r+q.
\end{align}
Now, since $p=Xqr+r+q$ we see that $a_{pr}=a_{qr}$.  Therefore
\begin{align}
\begin{split}
\bigg\{\frac{p}{qr} \bigg\}- \bigg\{\frac{a^{-1;r}_{qr}a_{pr}}{r}  \bigg\}&=\bigg\{\frac{Xqr+r+q}{qr}  \bigg\}- \bigg\{\frac{a^{-1;r}_{qr}a_{qr}}{r}  \bigg\}\\
&=\bigg\{\frac{1}{q}+\frac{1}{r} \bigg\}- \bigg\{\frac{1}{r} \bigg\}=\frac{1}{q}>0.
\end{split}
\end{align}
\end{proof}
The following is the main result of this section, which
is the same as Theorem \ref{introthm} upon identifying 
the integer $\epsilon$ with the number of exceptional singularities:
\begin{theorem}
\label{thm}
Let $g$ be the canonical Bochner-K\"ahler metric with reversed 
orientation on $\overline{\mathbb{CP}}^2_{(r,q,p)}$,
and assume that $1<r<q<p$.  If $r+q\geq p$ then 
\begin{align}
Ind(\overline{\mathbb{CP}}^2_{(r,q,p)},g)=2.
\end{align}
If $r+q<p$ then 
\begin{align}
Ind(\overline{\mathbb{CP}}^2_{(r,q,p)},g)=
\begin{cases}
2+2\epsilon -4\lfloor \frac{p}{qr} \rfloor &\text{   when $\{ \frac{p}{qr}\}<\{\frac{q^{-1;r}p}{r}\}$}\\
-2+2\epsilon-4\lfloor \frac{p}{qr} \rfloor &\text{   when $\{ \frac{p}{qr}\}>\{\frac{q^{-1;r}p}{r}\}$},
\end{cases}
\end{align}
where $\epsilon$ is the number of exceptional singularities, either $0$, $1$, or $2$.
\end{theorem}
Note that from Proposition $\ref{lee}$ the only instance when 
two exceptional singularities can occur is in the second case,
thus there are really only five distinct cases. All of these
cases do in fact occur, see Table \ref{casestable}.
\begin{table}
\caption{Cases in Theorem \ref{thm}}
\label{casestable}
    \begin{tabular}{ | l | l | c | p{2cm} |}
    \hline
    $(r,q,p)$ &  $\epsilon$ & $\{ \frac{p}{qr}\} - \{\frac{q^{-1;r}p}{r}\}$ \\ \hline
    $(3,7,11)$ & 0 & $< 0$  \\ \hline
    $(3,7,41)$ & 0 & $> 0$ \\ \hline
    $(3,7,25)$ & 1 & $< 0$  \\ \hline
    $(3,7,13)$ & 1 & $> 0$  \\ \hline
    $(3,7,31)$ & 2 & $> 0$  \\ \hline
    \end{tabular}
\end{table}
\begin{proof}[Proof of Theorem \ref{thm}]
Since $1<r<q<p$, there are three singularities. 
Furthermore, $\chi_{top} = 3$ and $\tau_{top} = -1$ (see \cite[Appendix B]{Dimca}), 
so the $\Gamma$-index is 
\begin{align}
\begin{split}
Ind&=
8 + N(-q^{-1;r}p,r)+N(-p^{-1;q}r,q)+N(-r^{-1;p}q,p)\\
= 8 &+ \bigg[C_{(-q^{-1;r}p,r)}+A(-q^{-1;r}p,r)-4\bigg(\bigg( \frac{-q^{-1;r}p}{r}\bigg)\bigg)-4\bigg(\bigg(\frac{-p^{-1;r}q}{r}\bigg)\bigg)\bigg]\\
&+\bigg[C_{(-r^{-1;q}p,q)}+A(-p^{-1;q}r,q)-4\bigg(\bigg( \frac{-r^{-1;q}p}{q}\bigg)\bigg)-4\bigg(\bigg(\frac{-p^{-1;q}r}{q}\bigg)\bigg)\bigg]\\
&+\bigg[C_{(-r^{-1;p}q,p)}+A(-r^{-1;p}q,p)-4\bigg(\bigg( \frac{-r^{-1;p}q}{p}\bigg)\bigg)-4\bigg(\bigg(\frac{-q^{-1;p}r}{p}\bigg)\bigg)\bigg],
\end{split}
\end{align}
recalling that $C_{(\alpha, \beta)}$ was defined above in \eqref{cab}. 
Then, using Rademacher's triple reciprocity for Dedekind sums \cite{Rademacher1}
\begin{align}
s(q^{-1;r}p,r)+s(p^{-1;q}r,q)+s(r^{-1;p}q,p) = -\frac{1}{4}+\frac{1}{12}\bigg(\frac{r}{pq}+\frac{q}{pr}+\frac{p}{qr}\bigg),
\end{align}
we see that
\begin{align}
\begin{split}
Ind&=8+[C_{(-q^{-1;r}p,r)}+C_{(-r^{-1;q}p,q)}+C_{(-r^{-1;p}q,p)}]\\
&\phantom{==}+48\bigg[\frac{1}{4}-\frac{1}{12}\bigg(\frac{r}{pq}+\frac{q}{pr}+\frac{p}{qr}\bigg)\bigg]\\
&\phantom{==}+4\bigg[\bigg(\bigg( \frac{q^{-1;r}p}{r}\bigg)\bigg)+\bigg(\bigg(\frac{p^{-1;r}q}{r}\bigg)\bigg)+\bigg(\bigg( \frac{r^{-1;q}p}{q}\bigg)\bigg)\bigg]\\
&\phantom{==}+4\bigg[\bigg(\bigg(\frac{p^{-1;q}r}{q}\bigg)\bigg)+\bigg(\bigg( \frac{r^{-1;p}q}{p}\bigg)\bigg)+\bigg(\bigg(\frac{q^{-1;p}r}{p}\bigg)\bigg)\bigg].
\end{split}
\end{align}
Now, using our reciprocity laws for sawtooth functions, 
Theorems $\ref{2}$ and $\ref{3}$, and the restrictions on the types of singularities admitted, Proposition $\ref{gr}$, we complete the proof for each case.  

When $r+q>p$:
\begin{align*}
Ind&=8+[-18]+48\bigg[\frac{1}{4}-\frac{1}{12}\bigg(\frac{r}{pq}+\frac{q}{pr}+\frac{p}{qr}\bigg)\bigg] +4\bigg[\frac{r}{pq}+\frac{q}{pr}+\frac{p}{qr}\bigg] =2.
\end{align*}

When $r+q=p$:
\begin{align*}
Ind&=8+[-14]+48\bigg[\frac{1}{4}-\frac{1}{12}\bigg(\frac{r}{pq}+\frac{q}{pr}+\frac{p}{qr}\bigg)\bigg] +4\bigg[\frac{r}{pq}+\frac{q}{pr}+\frac{p}{qr}-1\bigg] =2.
\end{align*}

When $r+q<p$ and $\{ \frac{p}{qr}\}<\{\frac{q^{-1;r}p}{r}\}$:
\begin{align*}
Ind&=8+[C_{(-q^{-1;r}p,r)}+C_{(-r^{-1;q}p,q)}+C_{(-r^{-1;p}q,p)}]\\
&\phantom{==}+48\bigg[\frac{1}{4}-\frac{1}{12}\bigg(\frac{r}{pq}+\frac{q}{pr}+\frac{p}{qr}\bigg)\bigg]+4\bigg[\frac{r}{pq}+\frac{q}{pr}+\frac{p}{qr}-\Big\lfloor \frac{p}{qr}\Big\rfloor\bigg]\\
&=20+[C_{(-q^{-1;r}p,r)}+C_{(-r^{-1;q}p,q)}+C_{(-r^{-1;p}q,p)}]-4\Big\lfloor \frac{p}{qr}\Big\rfloor\\
&=2+2\epsilon-4\Big\lfloor \frac{p}{qr}\Big\rfloor.
\end{align*}

When $r+q<p$ and $\{ \frac{p}{qr}\}>\{\frac{q^{-1;r}p}{r}\}$:
\begin{align*}
Ind&=8+[C_{(-q^{-1;r}p,r)}+C_{(-r^{-1;q}p,q)}+C_{(-r^{-1;p}q,p)}]\\
&\phantom{==}+48\bigg[\frac{1}{4}-\frac{1}{12}\bigg(\frac{r}{pq}+\frac{q}{pr}+\frac{p}{qr}\bigg)\bigg]+4\bigg[\frac{r}{pq}+\frac{q}{pr}+\frac{p}{qr}-1-\Big\lfloor \frac{p}{qr}\Big\rfloor\bigg]\\
&=16+[C_{(-q^{-1;r}p,r)}+C_{(-r^{-1;q}p,q)}+C_{(-r^{-1;p}q,p)}]-4\Big\lfloor \frac{p}{qr}\Big\rfloor\\
&=-2+2\epsilon-4\Big\lfloor \frac{p}{qr}\Big\rfloor.
\end{align*}
This completes the proof.
\end{proof}

We also state the following theorem, which 
gives the index in the cases when there are strictly less than three singularities. 
\begin{theorem}
Let $g$ be the canonical Bochner-K\"ahler metric with reversed 
orientation on $\overline{\mathbb{CP}}^2_{(r,q,p)}$.
When $1=r<q<p$ there are two singularities and
\begin{align}
Ind(\overline{\mathbb{CP}}^2_{(1,q,p)},g)=
\begin{cases}
2 &\text{   when $q=p-1$}\\
-4\lfloor \frac{p}{q} \rfloor+6 &\text{   when $p=eq-(q-1)$ and $q\neq p-1$}\\
-4\lfloor \frac{p}{q} \rfloor+4 &\text{   when $1\leq a_{pq} \leq q-2$ and $q>2$}.
\end{cases}
\end{align}
When $1=r=q<p$ there is one singularity and
\begin{align}
Ind(\overline{\mathbb{CP}}^2_{(1,1,p)},g)=-4p+12.
\end{align}
\end{theorem}
\begin{proof}
We have that
\begin{align}
\frac{1}{2}(15\chi_{top}+20\tau_{top})=8,
\end{align}
Since $1=r<q<p$ we know that $p>2$.  The first case follows from the 
reciprocity formula for $R^-(q,p)$ in Proposition $\ref{plus,minus}$.
The second case follows from $N(-1,p)=-4p+4$ in $\eqref{etc}$.
\end{proof}

\subsection{Proof of Theorem \ref{wpsthm}}
We first present a general result about $H^2(M,g)$ on certain 
self-dual K\"ahler orbifolds:
\begin{proposition}
\label{sdk}
Let $(M,g)$ be a compact self-dual K\"ahler orbifold
and assume that the set $M^{>0} = \{ p \in M, R(p) > 0 \}$
is non-empty. With the reversed orientation to make $g$ 
anti-self-dual, we have $H^2(M,g) = 0$. 
\end{proposition}
\begin{proof}
As mentioned in the Introduction, the metric 
$\tilde{g} = R^{-2} g$ is an Einstein metric, which is 
complete on components of $M^*$. 
If $Z \in S^2_0(\Lambda^2_+(T^*M))$ satisfies $\mathcal{D}_g^* Z = 0$, 
where $ \mathcal{D}^*_{g}$ is the adjoint of $\mathcal{D}_{g}$, 
then from conformal invariance $\mathcal{D}_{\tilde g}^* Z = 0$ 
when $Z$ is viewed as a $(1,3)$ tensor. 
We compute 
\begin{align}
|Z|_{\tilde{g}}^2 =  \tilde{g}^{ip} \tilde{g}^{jq} Z_{ijk}^{\phantom{ijk}l}
Z_{pql}^{\phantom{pql}k}
= R^4  {g}^{ip} {g}^{jq} Z_{ijk}^{\phantom{ijk}l} Z_{pql}^{\phantom{pql}k} = R^4 |Z|_{g}^2,
\end{align}
so we have 
\begin{align}
\label{zeqn}
|Z|_{\tilde{g}} = R^2 |Z|_{g}.
\end{align}
Let $M^*_1$ denote any non-trivial component of $M^*$. 
Since the metric $\tilde{g}$ is Einstein on $M^*_1$, 
from \cite[Proposition 5.1]{Itoh}, we have 
\begin{align}
\label{dds}
\mathcal{D}_{\tilde{g}} \mathcal{D}^*_{\tilde{g}} Z = \frac{1}{24} 
( 3 \nabla^*_{\tilde{g}} \nabla_{\tilde{g}} 
+ 2 R_{\tilde{g}})(2 \nabla^*_{\tilde{g}} \nabla_{\tilde{g}} + R_{\tilde{g}}) Z,
\end{align}
where $R_{\tilde{g}}$ is the (constant) scalar curvature of the 
Einstein metric $\tilde{g}$ on $M^*_1$.  
If $M^{> 0} = M$, then the maximum principle immediately implies that 
$Z = 0$. Otherwise, there is an nontrivial open component of $M^*$, which we 
again call $M^*_1$. The metric $\tilde{g}$ is a complete Einstein metric on 
$M^*_1$, and \eqref{zeqn} shows that $|Z|_{\tilde{g}} (x) = o(1)$ as 
$r \rightarrow 0$, where $r$ is the distance to the zero set of the 
scalar curvature. Viewed on the complete manifold $(M^*_1, \tilde{g})$,
$Z$ is then a decaying solution at infinity of $\eqref{dds}$. 
Since $R_{\tilde{g}}$ is a constant, a standard separation of variables 
argument (see for example \cite{Donn}) 
implies that $Z$ must decay faster than the inverse 
of any polynomial in the $\tilde{g}$ metric (it morever has
exponential decay). Equivalently, 
$|Z|_g = O(r^k)$ as $r \rightarrow 0$ for any $k > 0$. This implies that $Z$ has 
a zero of infinite order along the zero set of the 
scalar curvature. The unique continuation principle for 
elliptic operators (see \cite{Aron}) then implies that 
$Z$ is identically zero. 
\end{proof}
As a corollary, we obtain
\begin{corollary} 
\label{h2wps}
If $g$ is the canonical Bochner-K\"ahler metric with reversed 
orientation on $\overline{\mathbb{CP}}^2_{(r,q,p)}$, then $H^2(M,g) = 0$.
\end{corollary}
\begin{proof}
From \cite[Equation (2.32)]{DavidGauduchon},
the set $M^{> 0}$ is non-empty. So this follows immediately from Proposition \ref{sdk}.
\end{proof}

\begin{proof}[Proof of Theorem \ref{wpsthm}]
From Corollary \ref{h2wps}, $H^2(M,g) = 0$, so the actual
moduli space is locally isomorphic to $H^1/ H^0$. Depending upon 
the action of $H^0$, the moduli space
could therefore be of dimension $\dim(H^1)$, $\dim(H^1) - 1$, or $\dim(H^1) - 2$.
The result then follows immediately from the determination of 
$H^1(M,g)$ in Theorem \ref{thm}. 
\end{proof}

\subsection{Final remarks}
We end with a non-rigorous remark on the number-theoretic condition 
appearing in Theorem \ref{thm}.  Figure~\ref{plot} contains a plot 
of the function
\begin{align}
H(r,q,p(j)) = \Big\{ \frac{p}{qr} \Big\} - \Big\{\frac{q^{-1;r}p}{r}\Big\}
\end{align}
for $r = 3$ and $q = 7$, where the horizontal axis indexes the $j$th prime. 
The plot begins at the fifth prime, $11$, and ends with the $100$th prime $541$. 
This, along with other empirical examples, indicates that the cases 
$H > 0$ and $H < 0$ occur with the approximately the same frequency.

\numberwithin{figure}{section}
\begin{figure}
\includegraphics{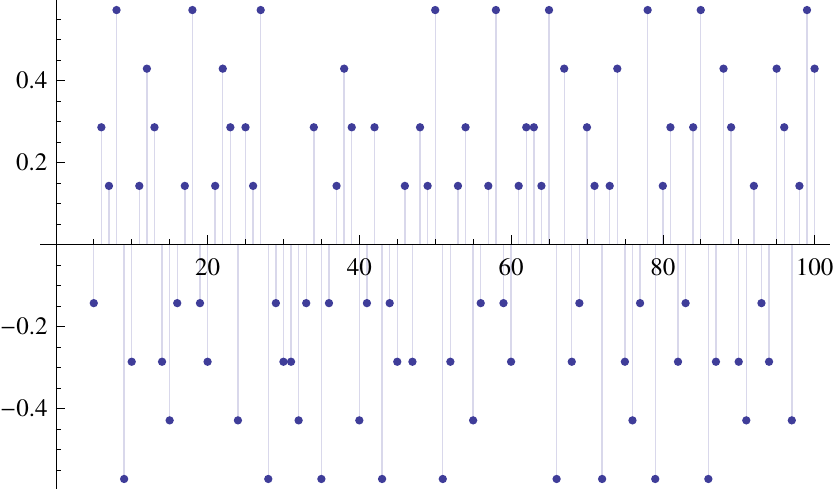}
\caption{$H(3,7,p(j))$}
\label{plot}
\end{figure}

\bibliography{ASD_references}

\end{document}